\documentclass[12pt,twoside]{amsart}
\usepackage{a4wide, amsmath,amsbsy,amsfonts,amssymb,mathabx,
stmaryrd,amsthm,mathrsfs,graphicx, amscd, tikz-cd, cancel}
\usepackage[normalem]{ulem}
\usepackage{soul}
\usetikzlibrary{matrix,arrows,decorations.pathmorphing}
\usepackage[active]{srcltx}
\usepackage[
	hypertexnames=false,
	hyperindex,
	pagebackref,  
	pdftex,
	breaklinks=true,
	bookmarks=false,
	colorlinks,
	linkcolor=blue,
	citecolor=red,
	urlcolor=red,
]{hyperref}

\usepackage[all]{xy}
\usepackage{subfig}
\usepackage{tikz}
\usetikzlibrary{shapes,arrows,shadows}
\usetikzlibrary{decorations.markings}
\usepackage{color}
\usepackage[normalem]{ulem}
\usepackage{hyperref}
\usepackage{wrapfig}
\usetikzlibrary{arrows}
\usepackage{pinlabel}
\long\def\symbolfootnote[#1]#2{\begingroup%
\def\thefootnote{\fnsymbol{footnote}}\footnote[#1]{#2}\endgroup}

\newtheorem{innercustomthm}{Theorem}
\newenvironment{customtheorem}[1]
  {\renewcommand\theinnercustomthm{#1}\innercustomthm}
  {\endinnercustomthm}

\newtheorem{theorem}{Theorem}[section]
\newtheorem{corollary}[theorem]{Corollary}
\newtheorem{proposition}[theorem]{Proposition}

\newtheorem{lemma}[theorem]{Lemma}

\theoremstyle{definition}
\newtheorem{definition}[theorem]{Definition}   
\newtheorem{remark}[theorem]{Remark}      
\newtheorem{remarks}[theorem]{Remarks}

\DeclareMathSymbol{\Alpha}{\mathalpha}{operators}{"41}
\DeclareMathSymbol{\Beta}{\mathalpha}{operators}{"42}
\DeclareMathSymbol{\Epsilon}{\mathalpha}{operators}{"45}
\DeclareMathSymbol{\Zeta}{\mathalpha}{operators}{"5A}
\DeclareMathSymbol{\Eta}{\mathalpha}{operators}{"48}
\DeclareMathSymbol{\Iota}{\mathalpha}{operators}{"49}
\DeclareMathSymbol{\Kappa}{\mathalpha}{operators}{"4B}
\DeclareMathSymbol{\Mu}{\mathalpha}{operators}{"4D}
\DeclareMathSymbol{\Nu}{\mathalpha}{operators}{"4E}
\DeclareMathSymbol{\Omicron}{\mathalpha}{operators}{"4F}
\DeclareMathSymbol{\Rho}{\mathalpha}{operators}{"50}
\DeclareMathSymbol{\Tau}{\mathalpha}{operators}{"54}
\DeclareMathSymbol{\Chi}{\mathalpha}{operators}{"58}
\DeclareMathSymbol{\omicron}{\mathord}{letters}{"6F}
\newcommand{\R}{\mathbb{R}}
\newcommand{\C}{\mathbb{C}}
\newcommand{\N}{\mathbb{N}}
\newcommand{\Z}{\mathbb{Z}}

\newcommand{\ZZ}{{\widehat{\mathbb Z}}}

\newcommand{\Q}{\mathbb{Q}}

\newcommand{\cB}{{\mathcal B}}
\newcommand{\bt}{\bullet}

\newcommand\Res{\operatorname{Res}}

\def\Aut{\operatorname{Aut}}
\def\Sp{\operatorname{Sp}}

\def\SL{\operatorname{SL}}
\def\PSL{\operatorname{PSL}}
\def\Inn{\operatorname{Inn}}

\def\Out{\operatorname{Out}}

\def\Spec{\operatorname{Spec}}

\def\P{{\mathbb P}}

\def\cP{{\mathcal P}}

\def\cL{{\mathcal L}}

\def\hL{\widehat{\mathcal L}}

\def\sr{\stackrel}

\def\cT{{\mathcal T}}

\def\wT{{\widehat{\mathcal T}}}

\def\cC{{\mathscr C}}

\def\cM{{\mathcal M}}
\def\cP{{\mathcal P}}
\def\cQ{{\mathcal Q}}
\def\hP{{\widehat\Pi}}
\def\hp{{\widehat\pi}}

\def\ccM{{\overline{\mathcal M}}}
\def\wM{{\widehat{\mathcal M}}}

\def\PM{\mathrm{P}\cM}
\def\CPM{\mathrm{P}\ccM}
\def\dd{\partial}

\def\l{{\lambda}}
\def\L{{\Lambda}}

\def\u{{\upsilon}}

\def\g{{\gamma}}
\def\G{{\Gamma}}
\def\U{{\Upsilon}}
\def\PG{{\mathrm{P}\Gamma}}
\def\PU{{\mathrm{P}\Upsilon}}
\def\kG{{\check\Gamma}}
\def\kU{{\check\Upsilon}}
\def\kPG{{\mathrm{P}\check\Gamma}}
\def\hG{{\widehat\Gamma}}
\def\hU{{\widehat\Upsilon}}
\def\hPG{{\mathrm{P}\widehat\Gamma}}
\def\hPU{{\mathrm{P}\widehat\Upsilon}}
\def\kPU{{\mathrm{P}\check\Upsilon}}

\def\d{{\delta}}
\def\s{{\sigma}}

\def\bt{\bullet}

\def\dS{{\partial S}}
\def\kC{{\check  C}}
\def\hC{{\widehat{C}}}
\def\cW{{\mathscr W}}

\def\ssm{\smallsetminus}
\def\ol{\overline}
\def\td{\tilde}
\def\wh{\widehat}

\def\ovr{\overrightarrow}
\def\hookra{\hookrightarrow}

\def\co{\colon\thinspace}

\begin{document}

\title{Notes on hyperelliptic mapping class groups}

\author[M. Boggi]{Marco Boggi}

\address{Instituto de Matem\'atica e Estat\'{\i}stica, Universidade Federal Fluminense \\
Rua Prof. M. W. de Freitas, S/N, Niter\'oi - RJ, 24210-200, Brazil.}
\email{marco.boggi@gmail.com}

\begin{abstract}    
Hyperelliptic mapping class groups are defined either as the centralizers of hyperelliptic involutions inside mapping class groups of 
oriented surfaces of finite type or as the inverse images of these centralizers by the natural epimorphisms between mapping class 
groups of surfaces with marked points. We study these groups in a systematic way.  An application of this theory is a counterexample to 
the genus $2$ case of a conjecture by Putman and Wieland on virtual linear representations of mapping class groups. In the last section, 
we study profinite completions of hyperelliptic mapping class groups: we extend the congruence subgroup property to the general class 
of hyperelliptic mapping class groups introduced above and then determine the centralizers of multitwists and of open subgroups in their 
profinite completions. 

\vskip .2cm
\noindent AMS Math Classification: 57K20, 14H10, 14F45, 14G32.
\end{abstract}

\maketitle

\section{Introduction}
A hyperelliptic surface is a connected orientable differentiable surface of finite type $S$ endowed with a hyperelliptic involution, that is to say
an automorphism $\u$ of order $2$ such that the quotient of $S$ by the action of this automorphism has genus $0$. 
The \emph{hyperelliptic mapping class group $\U(S)$ of $(S,\u)$} is the centralizer of $\u$ in the mapping class group $\G(S)$ 
of the surface $S$. For a finite set of points $\cP$ on $S$, the \emph{hyperelliptic mapping class group $\U(S,\cP)$ of the marked hyperelliptic surface 
$(S,\u,\cP)$} is defined to be the inverse image of $\U(S)$ via the natural epimorphism from the mapping class group $\G(S,\cP)$ 
of $(S,\cP)$ to $\G(S)$.

Note that, in the above definition of $\U(S,\cP)$, the marked points $\cP$ have no relation with the hyperelliptic involution $\u$ while, by definition,
the sets of punctures and boundary components of $S$ are preserved by the action of $\u$. In particular, for $\cP\neq\emptyset$, the  
hyperelliptic mapping class group $\U(S,\cP)$ is \emph{not} the centralizer in $\G(S,\cP)$ of some hyperelliptic involution.
This makes these objects more intractable than the hyperelliptic mapping class groups usually considered in the literature
where Birman--Hilden theory allows to reduce most problems about them to the genus $0$ case. 

More precisely, Birman--Hilden theory relates the hyperelliptic mapping class group of a hyperelliptic surface $(S,\u)$ (without marked points) to the mapping 
class group of the quotient surface $S_{/\u}:=S/\langle\u\rangle$ through the natural exact sequence (cf.\ \cite{BH}, \cite{BH2} and \cite{MH}):
\[1\to\langle\u\rangle\to\U(S)\to\G(S_{/\u}\ssm\cB_\u),\]
where $\cB_\u$ is the branch locus of the orbit map $S\to S_{/\u}$ and the image of $\U(S)$ in $\G(S_{/\u}\ssm\cB_\u)$ is a finite index subgroup
which can be explicitly (and easily) described. This often reduces problems about the hyperelliptic mapping class group $\U(S)$ to problems
about the genus $0$ mapping class group $\G(S_{/\u}\ssm\cB_\u)$.

A primary source of interest for these groups is the fact that a closed genus $2$ surface and a $1$-punctured, genus $1$ surface support, modulo isotopy,
a unique hyperelliptic involution $\u$, so that, for $S$ of one of these types, we have $\G(S)=\U(S)$. This possibly explains why hyperelliptic mapping class 
groups are usually considered for $S$ a closed surface or a surface with a single boundary component. 
However, it turns out that, for many questions, this approach is not sufficient. 

Consider for instance the subgroup $\U(S)_\g$ of elements of $\U(S)$ which preserve the isotopy class of a symmetric (i.e.\ invariant under the action
of the hyperelliptic involution $\u$) nonseparating simple closed curve $\g$ on $S$. Similarly to what happens for the mapping class group of $S$, this 
group can be described in terms of the hyperelliptic mapping class groups of the hyperelliptic surface $(S\ssm\g,\u)$, but this is of neither of the two
types above. 

In this paper, we describe the stabilizers of symmetric multicurves for hyperelliptic mapping class groups in full generality 
(cf.\ Section~\ref{complexsym}). Contrary to what happens for mapping class groups, the description of these stabilizers in terms of 
simpler mapping class groups is not straightforward. The problem is that the standard Birman short exact sequence, which describes the mapping class 
group of a surface with punctures in terms of the mapping class group of a surface with one less puncture, is not available here. 
In fact, the description of the hyperelliptic mapping class group of a hyperelliptic surface $(S,\u)$ endowed with a pair 
of symmetric punctures is rather subtle and more so for a marked hyperelliptic surface $(S,\u,\cP)$.
In the (unmarked) closed surface case, this problem was addressed by Brendle and Margalit in Section~3 of \cite{BM} (cf.\ Theorem~3.2 ibid.). 

In Section~\ref{HyperellipticBirman}, we present two Birman short exact sequences for hyperelliptic mapping class groups of surfaces endowed 
with a pair of symmetric punctures. 

For a marked hyperelliptic surface $(S,\u,\cP)$, where $S$ is hyperbolic and without boundary, let $Q\in S\ssm\cP$ be such that $Q\neq\u(Q)$
and let $S^\circ:=S\ssm\{Q,\u(Q)\}$. Let then $\U(S^\circ,\cP)_\circ$ be the index $2$ subgroup of $\U(S^\circ,\cP)$ consisting of those elements 
which do not swap the two punctures on $S^\circ$ obtained removing $Q$ and $\u(Q)$. The main result of Section~\ref{HyperellipticBirman}, and one of 
the main results of the paper, then is (cf.\ Theorem~\ref{Q} for a more general result):

\begin{customtheorem}{A}With the above notations, there is a natural short exact sequence:
\[1\to N_\circ\to\U(S^\circ,\cP)_\circ\to\U(S,\cP,Q)\to 1,\]
where $N_\circ$ is the normal subgroup generated by the Dehn twists about symmetric separating simple closed curves on $S^\circ$ bounding 
a disc which contains no marked points and only the two punctures obtained removing $Q$ and $\u(Q)$.
\end{customtheorem}

We do not know whether a proof of this theorem can be obtained by standard mapping class groups methods. 
We were only able to find a correct statement and a proof thanks to the geometric interpretation of 
hyperelliptic mapping class groups as fundamental groups of moduli stacks of hyperelliptic curves (cf.\ Section~\ref{geoint}).

Theorem~A is the version of Birman short exact sequence which turns out to be more useful for the applications 
considered in the rest of the paper. A simpler result holds for unmarked surfaces (cf.\ Theorem~\ref{HypBirman}). 
This is a somewhat more geometric version of Theorem~3.2 in \cite{BM}, where we also deal with non-closed surfaces. 

In Section~\ref{generators}, we describe generating sets for full and \emph{pure} (see Definition~\ref{pureHyp}) hyperelliptic mapping class groups, thus
generalizing some classical results by Birman, Hilden and Humphries (cf.\ \cite{BH} and \cite{Humphries}). 

In Section~\ref{homtype}, we study in more detail the complexes of symmetric curves introduced in Section~\ref{complexsym}. 
The main result is that, for a marked hyperelliptic surface $(S,\u,\cP)$ of genus at least $2$, the hyperelliptic nonseparating 
curve complex $C_\mathrm{ns}(S,\u,\cP)$ (cf.\ Definition~\ref{symmulticurvemarked}) is homotopic to a wedge of spheres of dimension $g-1$,
as in the classical case.

In Section~\ref{representations}, we provide counterexamples to a hyperelliptic analogue (cf.\ Section~\ref{HypProblem}) of a conjecture by 
Putman and Wieland (cf.\  Conjecture~1.2 in \cite{PW} and Section~\ref{PWconjecture}). Let $(S,\u,\cP)$ be a marked hyperelliptic closed surface
and fix a point $Q\in S\ssm\cP$. The point pushing map identifies the fundamental group $\pi_1(S\ssm\cP,Q)$ with a subgroup of the hyperelliptic
mapping class group $\U(S,\cP,Q)$. For a finite index subgroup $\U^\l$ of $\U(S,\cP,Q)$, let $\Pi^\l:=\pi_1(S\ssm\cP,Q)\cap\U^\l$ and let $S^\l\to S$ be
the, possibly ramified, covering of closed surfaces which compactifies the covering $(S\ssm\cP)^\l\to S\ssm\cP$ associated to $\Pi^\l$. 
Restriction of inner automorphisms induces a linear representation $\U^\l\to\Sp(H^1(S^\l,\Q))$. We then prove (cf.\ Corollary~\ref{NoHypPutWie}):

\begin{customtheorem}{B}For $g(S)\geq 2$ and $\sharp\,\cP\geq 4$, there is a finite index subgroup $\U^\l$ of 
$\U(S,\cP,Q)$ such that the associated linear representation $\U^\l\to\Sp(H^1(S^\l,\Q))$ has a finite nontrivial orbit.
\end{customtheorem}

Since, for $g(S)=2$, the hyperelliptic mapping class group $\U(S,\cP,Q)$ identifies with the mapping class group of the marked surface
$(S,\cP,Q)$, in particular, we get a counterexample in genus $2$ to the conjecture by Putman and Wieland mentioned above. 
This fills a gap, pointed out to me by Aaron Landesman and Daniel Litt, in the proof of Theorem~3.13 in the, now retracted, paper \cite{MMJ}.
 
In Section~\ref{profinite}, we study the profinite completions of the class of hyperelliptic mapping class groups introduced above. 
For a hyperbolic marked hyperelliptic surface $(S,\u,\cP)$, there is a natural faithful outer representation 
$\rho_{(S,\cP)}\co\U(S,\cP)\hookra\Out(\pi_1(S\ssm\cP))$ which, by the universal property of profinite completions, 
induces a continuous outer representation:
\[\hat{\rho}_{(S,\cP)}\co\hU(S,\cP)\to\Out(\hp_1(S\ssm\cP)).\]

The \emph{congruence subgroup problem for hyperelliptic mapping class groups} asks whether this representation is faithful as well.
In  \cite{Hyp} and \cite{congtop}, this question was answered positively for $S$ a closed surface of genus $\geq 1$.
In this paper, we extend the latter result to arbitrary hyperbolic marked hyperelliptic surfaces (cf.\ Theorem~\ref{conghyp}):

\begin{customtheorem}{C}Let $(S,\u,\cP)$ be a hyperbolic marked hyperelliptic surface. Then, the congruence subgroup 
property holds for the associated hyperelliptic mapping class group, that is to say, the representation $\hat{\rho}_{(S,\cP)}$ is faithful.
\end{customtheorem}

This is the first of a series of foundational results for profinite hyperelliptic mapping class groups which includes a proof
that hyperelliptic mapping class groups are "good" in the sense of a definition by Serre (cf.\ Theorem~\ref{goodhyp}) and
the computation of the centralizers of multitwists and of open subgroups of profinite hyperelliptic mapping class groups
(cf.\ Corollary~\ref{stabilizerprohypdescription}, Corollary~\ref{centprosymult} and Theorem~\ref{centerhyp}).
These latter results will play an important role in a forthcoming paper on higher genus Grothendieck-Teichm\"uller theory.
\bigskip

\noindent
{\bf Acknowledgements}. I thank Aaron Landesman and Daniel Litt for their comments on  \cite{MMJ}, which eventually led me to write this paper,
and two anonymous referees for their comments on a previous version which contributed to improve it.

\section{Hyperelliptic mapping class groups}
\subsection{Surfaces}\label{surfaces}
Unless otherwise stated, in this paper, a surface $S=S_{g,n}^k$ is a closed connected oriented differentiable surface of genus $g(S)=g$ 
from which $n(S)=n$ points and $k(S)=k$ open discs have been removed. A surface $S$ is \emph{hyperbolic} if its Euler characteristic $\chi(S)$ is negative.
For $k=0$ (resp.\ $n=0$), we let $S_{g,n}:=S_{g,n}^0$ (resp.\ $S_g^k:=S_{g,0}^k$). Also, we let $S_g:=S_{g,0}$.
The boundary of $S$ is denoted by $\dS$ and let $\ring{S}:=S\ssm\dS$, so that $\ring{S}\cong S_{g,n+k}$. 
We let $\ol{S}$ be the closed surface obtained from $S$ filling in the punctures and glueing a closed disc 
to each boundary component, so that $\ol{S}\cong S_{g}$.

A \emph{marked surface $(S,\cP)$} is a surface $S$ endowed with a set $\cP$ of distinct marked points. The marked surface $(S,\cP)$ is 
\emph{hyperbolic} if $S\ssm \cP$ has negative Euler characteristic. We denote by $\vec \cP$ the set $\cP$ with an order fixed on it. 
We then denote by $\cP_\circ$ the set of punctures of $S$ and by $\cP_\dd$ the set of connected components of $\dS$.

\subsection{Mapping class groups}\label{MCGr}
For a surface $S=S_{g,n}^k$, we let $\G(S)$ and $\PG(S)$ be respectively the mapping class group and 
the pure mapping class group associated to $S$. Note that the inclusion $\ring{S}\hookra S$ induces, by restriction, 
a natural homomorphism $\G(S)\hookra\G(\ring{S})$ whose image is the stabilizer of the partition of the punctures of $\ring{S}$ 
into those which have a boundary in $S$ and those which do not.


The mapping class group $\G(S,\cP)$ of the marked surface $(S,\cP)$ is the group of relative, with respect to the subset of points $\cP$, 
isotopy classes of orientation preserving diffeomorphisms of $S$. When isotopies are further required to preserve an order on the set $\cP$, 
we denote the corresponding mapping class group by $\G(S,\vec\cP)$. The \emph{pure} mapping class group $\PG(S,\cP)$ is the kernel 
of the natural homomorphism $\G(S,\cP)\to\Sigma_{\cP}\times\Sigma_{\cP_\circ}\times\Sigma_{\cP_\dd}$, where $\Sigma_\cP$, 
$\Sigma_{\cP_\circ}$ and $\Sigma_{\cP_\dd}$ are, respectively, the permutation groups on the sets $\cP$, $\cP_\circ$ and $\cP_\dd$.
There is then a natural isomorphism $\PG(S,\cP)\cong\PG(\ring{S}\ssm\cP)$.

\subsection{Mapping class groups of disconnected surfaces}\label{MCGdisc}
For some applications, it is useful to consider the mapping class group $\G(S,\cP)$ of a marked oriented differentiable \emph{disconnected} surface
$(S,\cP)$. This group can be easily described in terms of the mapping class groups of the connected components of $(S,\cP)$. 
Let $S=\coprod_{i=1}^h S^i$ be the decomposition in connected components and let $\cP^i:=\cP\cap S^i$, for $i=1,\ldots,h$. 
Let us fix a set of representatives $(S^{i_1},\cP^{i_1}),\ldots, (S^{i_k},\cP^{i_k})$ for the orbits 
of the action of $\G(S,\cP)$ on the set of marked connected components of $S$. Of course, two marked connected components are in the same orbit, 
if and only if, they are diffeomorphic and support the same number of marked points. Let us denote by $n_i$ the cardinality of the 
$\G(S,\cP)$-orbit of the component $(S^i,\cP^i)$ of $(S,\cP)$, for $i=1,\ldots,h$. There is then an isomorphism:
\begin{equation}\label{wreathMCG}
\G(S,\cP)\cong\prod_{j=1}^k\left(\G(S^{i_j},\cP^{i_j})\wr\Sigma_{n_{i_j}}\right),
\end{equation}
where $\G(S^{i_j},\cP^{i_j})\wr\Sigma_{n_{i_j}}$ denotes the wreath product of the mapping class group with the symmetric group on $n_{i_j}$ letters.

A similar description holds for the mapping class group $\G(S,\vec\cP)$ after replacing the $\G(S,\cP)$-orbits of the marked connected components 
of $(S,\cP)$ with their $\G(S,\vec\cP)$-orbits. Note that a connected component with a nontrivial marking has a trivial $\G(S,\vec\cP)$-orbit.

By definition, the \emph{pure mapping class group of $(S,\cP)$} is instead just the direct product:
\begin{equation}\label{wreathPMCG}
\PG(S,\cP)\cong\prod_{i=1}^h\PG(S^{i},\cP^{i}).
\end{equation}

\subsection{Relative mapping class groups}\label{MCGboundary}
For a marked hyperbolic surface $(S=S_{g,n}^k,\cP)$, we let $\G(S,\cP,\dS)$ be the \emph{relative mapping class group of the marked surface 
with boundary $(S,\cP,\dS)$}, that is to say the group of relative, with respect to the subspace $\cP\cup\dS$, isotopy classes of orientation preserving 
diffeomorphisms of $S$. Let $\delta_1,\ldots,\delta_k$ be the connected components of the boundary $\dS$. Then, for a hyperbolic marked surface
$(S,\cP)$, the group $\G(S,\cP,\dS)$ is described by the natural short exact sequence:
\[1\to\prod_{i=1}^k\tau_{\d_i}^\Z\to\G(S,\cP,\dS)\to\G(S,\cP)\to 1.\]

The relative \emph{pure} mapping class group $\PG(S,\cP,\dS)$ is the subgroup of $\G(S,\cP,\dS)$ consisting of those elements which admit 
representatives fixing pointwise the sets of marked points, punctures and boundary components of $S$. For a hyperbolic marked surface
$(S,\cP)$, it is described by the short exact sequence:
\[1\to\PG(S,\cP,\dS)\to\G(S,\cP,\dS)\to\Sigma_\cP\times\Sigma_n\times\Sigma_k\to 1.\]

\subsection{Hyperelliptic surfaces}\label{hypsurfaces}
A \emph{hyperelliptic surface $(S,\u)$} is the data consisting of a surface $S$ (cf.\ Section~\ref{surfaces})
together with a hyperelliptic involution $\u$ on $S$ (i.e.\ a 
self-diffeomorphism $\u$ of $S$ such that $\u^2=\mathrm{id}_S$ and the quotient surface $S/\langle\u\rangle$ has genus $0$).
A Weierstrass point, resp.\ puncture, resp.\ boundary component of $(S,\u)$ is a point, resp.\ puncture, resp.\ boundary component, which is preserved
by the hyperelliptic involution $\u$. We then denote by $\cW_\cP(S)$, resp.\ $\cW_\circ(S)$, resp.\ $\cW_\dd(S)$, the set of Weierstrass points, resp.\ 
punctures, resp.\ boundary components on $(S,\u)$.
A \emph{marked hyperelliptic surface $(S,\u,\cP)$} is a hyperelliptic surface $(S,\u)$ endowed with a set $\cP$ 
of distinct marked points.

\subsection{Hyperelliptic mapping class groups}\label{HMCGr}
For a hyperelliptic surface $(S,\u)$, let us also denote by $\u$ the image of the involution $\u$ in the mapping class group $\G(S)$ 
(note that this image is trivial if and only if $g(S)=0$ and $S$ is not hyperbolic).
The \emph{hyperelliptic mapping class group} $\U(S)$ of $(S,\u)$ is then defined to be the centralizer of $\u$ in $\G(S)$.

Let $S_{/\u}:=S/\langle\u\rangle$ be the quotient surface, $p_\u\co S\to S_{/\u}$ be the orbit map and denote by $\cB_\u=p_\u(\cW_\cP(S))$ 
the branch locus of $p_\u$. Let then $\G(S_{/\u},\cB_\u)_{p_\u}$ be the subgroup of the mapping class group $\G(S_{/\u},\cB_\u)$ consisting of 
those elements which preserve the partitions of the sets of punctures and boundary components of $S_{/\u}$ into those which are in 
the image of $\cW_\circ(S)$ and $\cW_\dd(S)$ by $p_\u$ and those who are not. 

For a hyperbolic hyperelliptic surface $(S,\u)$, there is then a natural short exact sequence (cf.\ Corollary~12 in \cite{MH} and Proposition~3.1 in \cite{MW}):
\begin{equation}\label{HypEx}
1\to\langle\u\rangle\to\U(S)\to\G(S_{/\u},\cB_\u)_{p_\u}\to 1,
\end{equation}
and a natural surjective representation:
\begin{equation}\label{Weierstrassrepr}
\rho_{\cW}(S)\co\U(S)\to\Sigma_{\cW_\cP(S)}\times\Sigma_{\cW_\circ(S)}\times\Sigma_{\cW_\dd(S)},
\end{equation}
where $\Sigma_{\cW_\cP(S)}$, $\Sigma_{\cW_\circ(S)}$ and $\Sigma_{\cW_\dd(S)}$ are, respectively, the symmetric groups
on the sets $\cW_\cP(S)$, $\cW_\circ(S)$ and $\cW_\dd(S)$. 

\begin{definition}\label{pureHyp} The \emph{pure hyperelliptic mapping class group of a hyperbolic hyperelliptic surface $(S,\u)$} 
is defined to be the intersection inside the mapping class group $\G(S)$:
\[\PU(S):=\ker\rho_{\cW}(S)\cap\PG(S).\]
\end{definition}

\begin{remark}Note that, with this definition, $\PU(S)\subsetneq\G(S)$ even for $S=S_2$ and $S=S_{1,1}$.
\end{remark}

The short exact sequence~(\ref{HypEx}) takes a simpler form for the pure hyperelliptic mapping class group (this in part motivated the definition):
\begin{equation}\label{PHypEx}
1\to\langle\u\rangle\cap\PU(S)\to\PU(S)\to\PG(S_{/\u},\cB_\u)\to 1.
\end{equation}
Observe that $\PG(S_{/\u},\cB_\u)\cong\PG(S_{/\u}\ssm\cB_\u)$. In particular, if $S$ contains at least two symmetric punctures 
or boundary components, then $\langle\u\rangle\cap\PU(S)=\{1\}$ and there is a series of natural isomorphisms:
\[\PU(S)\cong\PG(S_{/\u},\cB_\u)\cong\PG(S_{/\u}\ssm\cB_\u).\]

\begin{remark}Given a hyperbolic hyperelliptic surface $(S,\u)$, let $(\td{S},\u)$ be the hyperelliptic surface obtained removing from $S$ 
all Weierstrass points and all Weierstrass boundary components. There is then a natural injective homomorphism $\U(S)\to\U(\td{S})$, 
induced by restriction of diffeomorphisms, whose image is the stabilizer, under the representation 
$\rho_{\cW}(\td{S})\co\U(\td{S})\to\Sigma_{\cW_\circ(\td{S})}$, of the partition of $\cW_\circ(\td{S})$ 
induced by the natural embedding $\td{S}\hookra S$. In particular, for $(S,\u)$ a hyperelliptic surface with
$\cW_\circ(S)=\cW_\dd(S)=\emptyset$, restriction of diffeomorphisms induces a natural isomorphism $\U(S)\cong\U(\td{S})$.
\end{remark}

\subsection{Hyperelliptic mapping class groups of a marked hyperelliptic surface} 
The hyperelliptic mapping class groups $\U(S,\cP)$ and $\U(S,\vec{\cP})$ of the marked hyperelliptic surface  $(S,\u,\cP)$ 
are defined to be, respectively, the inverse images of $\U(S)$ by the natural epimorphisms $\G(S,\cP)\to\G(S)$ and $\G(S,\vec{\cP})\to\G(S)$.

The \emph{pure hyperelliptic mapping class group $\PU(S,\cP)$} is the inverse image of $\PU(S)$ by 
the natural epimorphism $\PG(S,\cP)\to\PG(S)$. 

For a hyperbolic marked hyperelliptic surface $(S,\u,\cP)$, with these definitions, for a point $Q\in S\ssm\cP$, 
we get, by restriction of the usual Birman short exact sequence for the mapping class group of the marked surface $(S,\cP\cup\{Q\})$, 
the short exact sequences:
\begin{equation}\label{BirmanExact}\begin{array}{ll}
&1\to\pi_1(S\ssm\cP,Q)\to\PU(S,\cP\cup\{Q\})\to\PU(S,\cP)\to 1,\\
\\
&1\to\pi_1(S\ssm\cP,Q)\to\U(S,\ovr{\cP\cup\{Q\}})\to\U(S,\vec\cP)\to 1\\
\mbox{and}\hspace{0.5cm}&\\
&1\to\pi_1(S\ssm\cP,Q)\to\U(S,\cP,Q)\to\U(S,\cP)\to 1,
\end{array}
\end{equation}
where we let $\U(S,\cP,Q)$ be the stabilizer of the point $Q$ for the action of the hyperelliptic mapping class group $\U(S,\cP\cup\{Q\})$
on the set of marked points $\cP\cup\{Q\}$.


\begin{remark}Note that, if the set of points $\cP$ is preserved by the hyperelliptic involution $\u$, then $(S\ssm\cP,\u)$ is a 
hyperelliptic surface. However, unlike the case of mapping class groups, for $\cP\neq\emptyset$ and $(S,\cP)$ hyperbolic, we have:
\[\U(S,\cP)\supsetneq\U(S\ssm\cP)\hspace{1cm}\mbox{and}\hspace{1cm}\PU(S,\cP)\supsetneq\PU(S\ssm\cP).\]
In fact, by definition, $\U(S\ssm\cP)$ and $\PU(S\ssm\cP)$ are contained in the centralizer of the hyperelliptic involution $\u$ in $\G(S\ssm\cP)$ while 
both $\U(S,\cP)$ and $\PU(S,\cP)$ contain the kernel of the epimorphism $\PG(S,\cP)\cong\PG(S\ssm\cP)\to\PG(S)$ which does not commute with $\u$.
\end{remark}

For $m\geq 2$, there is a natural representation $\rho_{[m]}\co\U(S,\cP)\to\Sp(H_1(\ol{S},\Z/m))$ whose kernel we denote by $\U(S,\cP)^{[m]}$
and call the \emph{abelian level of order $m$ of} $\U(S,\cP)$. The pure hyperelliptic mapping class group has then the following characterization: 

\begin{proposition}\label{purehyp} For $S$ a closed surface of genus $\geq 2$ or $S=S_{1,1}$, we have:
\[\PU(S,\cP)=\U(S,\vec{\cP})^{[2]}.\]
\end{proposition}

\begin{proof}Proposition~3.3 and Remark~3.4 in \cite{Hyp} imply that, for an unmarked hyperelliptic surface $(S,\u)$ as in the hypothesis,
since both subgroups $\PU(S)$ and $\U(S)^{[2]}$ of $\U(S)$ contain the hyperelliptic involution $\u$, there holds $\PU(S)=\U(S)^{[2]}$.

For the marked case, let $\pi_{\vec{\cP}}\co\U(S,\vec{\cP})\to\U(S)$ be the natural epimorphism. Then, we have
that $\PU(S,\cP)=\pi_{\vec{\cP}}^{-1}(\PU(S))$ and $\U(S,\vec{\cP})^{[2]}=\pi_{\vec{\cP}}^{-1}(\U(S)^{[2]})$.
Hence, the conclusion follows also in this case.
\end{proof}

\subsection{Hyperelliptic mapping class groups of disconnected surfaces}\label{disconnected}
A \emph{disconnected} hyperelliptic surface $(S,\u)$ is an oriented differentiable surface $S$ endowed with an involution $\u$ such that 
all connected components of $S_{/\u}:=S/\langle\u\rangle$ have genus $0$.  We say that $S$ is \emph{hyperbolic} if all its connected 
components are hyperbolic. The sets $\cW_\cP(S)$, $\cW_\circ(S)$ and $\cW_\dd(S)$ of Weierstrass points, punctures 
and boundary components of $(S,\u)$ are defined as in Section~\ref{HMCGr}.

The \emph{hyperelliptic mapping class group $\U(S)$ of the disconnected surface $(S,\u)$} is then defined to be the centralizer of $\u$ 
in the mapping class group $\G(S)$ (cf.\ Section~\ref{MCGdisc}). 

Let us observe that the hyperelliptic involution $\u$ acts on the set of connected components of $S$ and that, 
if a component is not fixed by this action, it has genus $0$. As in Section~\ref{MCGdisc}, let $S=\coprod_{i=1}^h S^i$ be 
the decomposition in connected components and let $S^{i_1},\ldots, S^{i_k}$ be  a set of representatives for their $\G(S)$-orbits.
Let us further assume that $\u(S^{i_j})=S^{i_j}$, for $j=1,\ldots,t$, and 
$\u(S^{i_j})\neq S^{i_j}$, for $j=t+1,\ldots k$. 
From the isomorphism~(\ref{wreathMCG}), it immediately follows that there is an isomorphism:
\begin{equation}\label{wreathHMCG}
\U(S)\cong\prod_{j=1}^t\left(\U(S^{i_j})\wr\Sigma_{n_{i_j}}\right)\times\prod_{j=t+1}^k\left(\G(S^{i_j})\wr\Sigma_{n_{i_j}/2}\right).
\end{equation}

For a hyperbolic disconnected hyperelliptic surface $(S,\u)$, the hyperelliptic mapping class group $\U(S)$ then also fits in the natural short exact sequence:
\[1\to\langle\u\rangle\to\U(S)\to\G(S_{/\u},\cB_\u)_{p_\u}\to 1,\]
where $p_\u\co S\to S_{\u}$ is the natural map, $\cB_\u$ its branch locus and the subgroup $\G(S_{/\u},\cB_\u)_{p_\u}$ of $\G(S_{/\u},\cB_\u)$ 
is defined in the same way as in Section~\ref{HMCGr}. It follows that, there is also a natural surjective representation:
\[\rho_{\cW}(S)\co\U(S)\to\Sigma_{\cW_\cP(S)}\times\Sigma_{\cW_\circ(S)}\times\Sigma_{\cW_\dd(S)}.\]

The \emph{pure} hyperelliptic mapping class group of a disconnected hyperbolic hyperelliptic surface $(S,\u)$ 
is then defined to be the intersection in the mapping class group $\G(S)$:
\[\PU(S):=\ker\rho_{\cW}(S)\cap\PG(S).\]
There is a natural short exact sequence:
\[1\to\langle\u\rangle\cap\PU(S)\to\PU(S)\to\PG(S_{/\u},\cB_\u)\to 1.\]

In order to obtain an explicit description of $\PU(S)$, let us put the decomposition of $S$ in connected components in the form:
\begin{equation}\label{dec}
S=\left(\coprod_{i=1}^l S^i\right)\coprod\left(\coprod_{i=l+1}^{r} S^i\right)\coprod\left(\coprod_{i=l+1}^r \u(S^i)\right),
\end{equation}
where $\u(S^i)=S^i$ if and only if $1\leq i\leq l$ and $r:=\frac{h+l}{2}$. There is then an isomorphism:
\begin{equation}\label{PHMCG}
\PU(S)\cong\prod_{i=1}^l\PU(S^{i})\times\prod_{i=l+1}^r \PG(S^{i}).
\end{equation}

\subsection{Hyperelliptic mapping class groups of marked disconnected surfaces}\label{markdisconnected}
As in Section~\ref{HMCGr}, for a marked disconnected hyperelliptic surface $(S,\u,\cP)$, the hyperelliptic mapping class groups 
$\U(S,\cP)$, $\U(S,\vec\cP)$ and $\PU(S,\cP)$ are defined to be the inverse images of $\U(S)$ and  $\PU(S)$ under the natural homomorphisms 
$\G(S,\cP)\to\G(S)$, $\G(S,\vec\cP)\to\G(S)$ and $\PG(S,\cP)\to\PG(S)$, respectively. 

An explicit description of the hyperelliptic mapping class groups associated to a marked disconnected surface is, in general, complicated,
but it is simple in the following case:

\begin{proposition}\label{hypfixed}Let us assume that all connected components of $S$ are fixed by the hyperelliptic involution $\u$.
With the notations of Section~\ref{MCGdisc}, there is then an isomorphism:
\begin{equation}\label{wreathHMCG2}
\U(S,\cP)\cong\prod_{j=1}^k\left(\U(S^{i_j},\cP^{i_j})\wr\Sigma_{n_{i_j}}\right),
\end{equation}
a similar one for $\U(S,\vec\cP)$ and an isomorphism:
\begin{equation}\label{PHMCG2}
\PU(S,\cP)\cong\prod_{i=1}^h\PU(S^{i},\cP^i).
\end{equation}
\end{proposition}

\begin{proof}This is an easy consequence of the isomorphisms~(\ref{wreathHMCG}) and~(\ref{PHMCG}) 
and the Birman exact sequences~(\ref{BirmanExact}).
\end{proof}

\subsection{Relative hyperelliptic mapping class groups}\label{HMCGboundary}
The \emph{relative hyperelliptic mapping class group $\U(S,\dS)$ of the hyperelliptic surface with boundary $(S,\u,\dS)$}
is the centralizer of the hyperelliptic involution $\u$ in the relative mapping class group $\G(S,\dS)$.
For $(S,\u)$ hyperbolic, there is also a natural surjective representation:
\[\rho_{\cW}(S)\co\U(S)\to\Sigma_{\cW_\cP(S)}\times\Sigma_{\cW_\circ(S)}\times\Sigma_{\cW_\dd(S)}.\]
The relative \emph{pure} hyperelliptic mapping class group of the hyperbolic hyperelliptic surface $(S,\u)$ is then the intersection:
\[\PG(S,\dS):=\ker\rho_{\cW}(S)\cap\PG(S,\dS).\]

Let $\delta_1,\ldots,\delta_{h}$ be the connected components of the boundary $\dS$ which are fixed by $\u$. The remaining boundary 
components can then be divided into pairs $\delta_{h+1},\u(\delta_{h+1}),\ldots\delta_k,\u(\delta_k)$.
The groups $\U(S,\dS)$ and $\PU(S,\dS)$ are described by the natural short exact sequences:
\[\begin{array}{ll}
&1\to\prod_{i=1}^k\Z\to\U(S,\dS)\to\U(S)\to 1\\
\mbox{and}\hspace{0.5cm}&\\
&1\to\prod_{i=1}^k\Z\to\PU(S,\dS)\to\PU(S)\to 1,
\end{array}\]
where the left hand maps send the $i$-th standard basis vector of $\prod_{i=1}^k\Z$ to the Dehn twist $\tau_{\delta_i}$, for $1\leq i\leq h$, and to 
the product of Dehn twists $\tau_{\delta_i}\cdot\tau_{\u(\delta_i)}$, for $h+1\leq i\leq k$.

The relative hyperelliptic mapping class group $\U(S,\cP,\dS)$ and relative pure hyperelliptic mapping class group $\PU(S,\cP,\dS)$ 
of a marked hyperelliptic surface  $(S,\u,\cP)$ are defined to be the inverse images of $\U(S,\dS)$ and $\PU(S,\dS)$, respectively, 
via the natural epimorphisms $\G(S,\cP,\dS)\to\G(S,\dS)$ and $\PG(S,\cP,\dS)\to\PG(S,\dS)$.

\subsection{The complex of symmetric curves on $(S,\u)$}\label{complexsym}
For a hyperbolic surface $S$, the complex of curves $C(S)$ is defined to be the (abstract) simplicial complex whose $k$-simplices consist of unordered 
$(k+1)$-tuples $\{[\g_0],\ldots,[\g_k]\}$ of distinct isotopy classes of pairwise disjoint non-peripheral simple closed curves on $S$.
For a hyperelliptic surface $(S,\u)$, we then define the \emph{complex of symmetric curves} $C(S,\u)$ to be the subcomplex of $C(S)$ consisting 
of simplices fixed by the hyperelliptic involution $\u$. More intrinsically and more generally, this can be defined as follows:

\begin{definition}\label{symmulticurvemarked}\begin{itemize}\leavevmode
\item A \emph{multicurve} is a set $\s=\{\g_0,\ldots,\g_k\}$ of mutually disjoint simple closed curves on $S$.
A multicurve $\s$ on $S$ is \emph{essential} if its elements are in distinct isotopy classes and no element of $\s$ is peripheral. 
We say that a multicurve $\s$ on a hyperelliptic surface $(S,\u)$ is \emph{symmetric} if $\u(\s)$ is isotopic to $\s$.


\item The \emph{complex of symmetric curves $C(S,\u,\cP)$ of the 
marked hyperelliptic surface $(S,\u,\cP)$} is defined to be the (abstract) simplicial complex whose 
simplices are isotopy classes of essential multicurves on $S\ssm\cP$ such that their image on $S$ is symmetric. 
The \emph{complex of symmetric nonseparating curves} $C_\mathrm{ns}(S,\u,\cP)$ is then the subcomplex consisting 
of those simplices $\s\in C(S,\u,\cP)$ such that $S\ssm\s$ is connected.
\end{itemize}
\end{definition}

\begin{remark}Note that any essential multicurve on $S\ssm\cP$ which becomes null-homotopic on $S$ is, in particular, (trivially) symmetric. 
\end{remark}

There is a natural action of the mapping class group $\G(S,\cP)$ of a marked hyperbolic surface $(S,\cP)$ on the curve complex $C(S\ssm\cP)$.
The stabilizer $\G(S,\cP)_\s$ of the simplex $\s$ is 
described by the exact sequence:
\[1\to\prod_{\g\in\s}\tau_\g^{\Z}\to\G(S,\cP)_\s\to\G(S\ssm\s,\cP),\]
where the mapping class group $\G(S\ssm\s,\cP)$ of the marked (possibly) disconnected surface $(S\ssm\s,\cP)$ is described by the 
isomorphism~(\ref{wreathMCG}) and the image of $\G(S,\cP)_\s$ in this group has finite index. In fact, it consists of those elements of 
$\G(S\ssm\s,\cP)$ which preserve the gluing data on $S\ssm\s$ induced by its embedding in $S$ as encoded in the dual
graph associated to $\s$ (the finite connected graph obtained collapsing to a point each connected component of the complement of a tubular 
neighborhood of $\s$ in $S$ and then collapsing to a segment each connected component of the tubular neighborhood).

Let us  denote by $\vec\s$ the simplex $\s$ endowed with an order and an orientation on each of its elements. 
The stabilizer $\PG(S,\cP)_\s$ of a simplex $\s\in C(S\ssm\cP)$ 
for the action of the pure mapping class group $\PG(S,\cP)$ is then described by the short exact sequences:
\begin{equation}\label{stabPMCG}\begin{array}{ll}
&1\to\PG(S,\cP)_{\vec\s}\to\PG(S,\cP)_\s\to\Sigma_{\s^\pm}\\
\mbox{and}\hspace{0.5cm}&\\
&1\to\prod_{\g\in\s}\tau_\g^{\Z}\to\PG(S,\cP)_{\vec\s}\to\PG(S\ssm\s,\cP)\to 1,
\end{array}
\end{equation}
where $\Sigma_{\s^\pm}$ is the symmetric group on the set of oriented simple closed curves $\s^\pm:=\s^+\cup\s^-$ and
the group $\PG(S\ssm\s,\cP)$ is described by the isomorphism~(\ref{wreathPMCG}).
 
From these descriptions, we deduce a similar description for the stabilizers of the action of the hyperelliptic mapping class groups $\U(S,\cP)$ and
$\PU(S,\cP)$ on the complex of symmetric curves $C(S,\u,\cP)$ of a marked hyperelliptic surface $(S,\u,\cP)$. 

Let us observe that, for $\s\in C(S,\u,\cP)$, we can choose a set of disjoint representatives on $S\ssm\cP$, which we also denote by $\s$, such that  
there holds $\u(\s)=\s$, so that $(S\ssm\s,\u)$ is a, possibly disconnected, hyperelliptic surface. 
The stabilizer $\U(S,\cP)_\s$ of a simplex $\s\in C(S,\u,\cP)$ is then described by the exact sequence:
\begin{equation}\label{stabHyp}
1\to\prod_{\g\in\s}\tau_\g^{\Z}\to\U(S,\cP)_\s\to\U(S\ssm\s,\cP),
\end{equation}
where the hyperelliptic mapping class group $\U(S\ssm\s,\cP)$ of the marked, possibly disconnected, hyperelliptic 
surface $(S\ssm\s,\u,\cP)$ is described in sections~\ref{disconnected} and \ref{markdisconnected} and the image of $\U(S,\cP)_\s$ 
in $\U(S\ssm\s,\cP)$ has finite index.

If $\s=\{\g\}$ is a $0$-simplex, let us denote by $Q_+,Q_-$ the pair of punctures on $S\ssm\g$ bounded by $\g$ and by 
$\U(S\ssm\g,\cP)_{\{Q_+,Q_-\}}$ the stabilizer of this pair of punctures for the action of the hyperelliptic mapping class group $\U(S\ssm\g,\cP)$
on the set of punctures of $S\ssm\g$. Then, there is a short exact sequence:
\begin{equation}\label{stabHyp0}
1\to\tau_\g^{\Z}\to\U(S,\cP)_\s\to\U(S\ssm\g,\cP)_{\{Q_+,Q_-\}}\to 1.
\end{equation}

Let $\s_{ns}$ and $\s_{s}$ be the subsets of $\s$ consisting, respectively, of nonseparating and of separating curves.
Then, the stabilizer $\PU(S,\cP)_\s$ is described by the exact sequences:
\begin{equation}\label{stabPHyp}\begin{array}{ll}
&1\to\PU(S,\cP)_{\vec\s}\to\PU(S,\cP)_\s\to\Sigma_{\s^\pm}\\
\mbox{and}\hspace{0.5cm}&\\
&1\to\prod_{\g\in\s_{ns}}\tau_\g^{2\Z}\times\prod_{\g\in\s_{s}}\tau_\g^{\Z}\to\PU(S,\cP)_{\vec\s}\to\PU(S\ssm\s,\cP)\to 1,
\end{array}
\end{equation}
where the pure hyperelliptic mapping class group $\PU(S\ssm\s,\cP)$ of the marked, possibly disconnected, hyperelliptic 
surface $(S\ssm\s,\u,\cP)$ above is described in sections~\ref{disconnected} and \ref{markdisconnected}.

\section{Geometric interpretation of hyperelliptic mapping class groups}\label{geoint}
\subsection{Moduli stacks of complex curves of a fixed topological type}
Let $S=S_{g,n}$ be a hyperbolic surface without boundary. We can then associate to $S$ the moduli stack $\cM(S)$ of $n$-punctured, genus $g$
smooth curves. The stack $\cM(S)$ represents the functor of smooth curves $\cC\to X$ such that each geometric fiber of the family has a complex 
model homeomorphic to the surface $S$. This is a smooth, irreducible, Deligne-Mumford (briefly DM) stack over $\Spec\Z$.
Fixing an order on the set of punctures of the geometric fibers of the curves $\cC\to X$, defines an \'etale Galois covering of DM stacks 
$\mathrm{P}\cM(S)\to\cM(S)$.

The moduli stack $\cM(S)$ admits a natural compactification $\ccM(S)$, called the DM compactification. This is the stack which represents $n$-punctured,
genus $g$ stable curves, that is to say, flat nodal curves $\cC\to X$ such that each geometric fiber of the family has a complex model homeomorphic to 
the, possibly singular, topological surface obtained from $S$ collapsing to points the connected components of a, possibly empty, essential multicurve. 

The stack $\ccM(S)$ is a smooth, proper DM stack also defined over $\Spec\Z$ and, fixing an order on the set of punctures of the fibers 
of the relative curves $\cC\to X$, defines an \'etale Galois covering of DM stacks $\mathrm{P}\ccM(S)\to\ccM(S)$.

In this paper, we are essentially interested in the complex theory of these objects. Therefore, without further specification, from now on, we will 
denote by $\cM(S)$, $\mathrm{P}\cM(S)$, $\ccM(S)$ and $\mathrm{P}\ccM(S)$ the corresponding complex DM stacks.

The mapping class groups of the surface $S$ which we introduced in Section~\ref{MCGr} have a natural interpretation as the topological fundamental
groups of the above complex DM stacks. Indeed, a marking $S\to C$ of an $n$-punctured, genus $g$ complex smooth curve determines isomorphisms 
$\G(S)\cong\pi_1(\cM(S),[C])$ and $\PG(S)\cong\pi_1(\mathrm{P}\cM(S),[C])$.

\subsection{Moduli stacks of marked complex curves}
For a marked hyperbolic  surface $(S,\cP)$, let $\cM(S,\cP)$ be the moduli stack which represents the functor 
of smooth curves $\cC\to X$, endowed with sections $s_Q\co X\to\cC$, labeled by the points $Q\in\cP$, such that, for every geometric point $x\in X$, 
the fiber $(\cC_x,\{s_Q(x)\}_{Q\in\cP})$ has a complex model homeomorphic to the marked surface $(S,\cP)$. The moduli stack $\cM(S,\vec\cP)$
is defined by fixing an order on $\cP$ (and then on the set of sections).
These are smooth, irreducible, DM stacks over $\Spec\Z$. There is a natural smooth morphism $\cM(S,\cP)\to\cM(S)$ 
(resp.\ $\cM(S,\vec\cP)\to\cM(S)$) whose fiber $\cM(S,\cP)_{[C]}$ (resp.\  $\cM(S,\vec\cP)_{[C]}$), for a geometric point $[C]\in\cM(S)$, identifies
with the configuration spaces of $n$ distinct (resp.\ distinct and ordered) points on $C$. 

Let then $\mathrm{P}\cM(S,\cP)$ be the fibered product of $\mathrm{P}\cM(S)$ and $\cM(S,\vec\cP)$ over $\cM(S)$.
There is a natural isomorphism of DM stacks $\mathrm{P}\cM(S,\cP)\cong\mathrm{P}\cM(S\ssm\cP)$.

The DM compactification $\ccM(S,\cP)$ is the moduli stack of flat nodal curves $\cC\to X$, endowed with disjoint sections $s_Q\co X\to\cC$, for $Q\in\cP$, 
such that, for every geometric point $x\in X$, the fiber $\cC_x\ssm\{s_Q(x)\}_{Q\in\cP}$ has a complex model homeomorphic to
the surface obtained from $S\ssm\cP$ collapsing to points the connected components of a, possibly empty, essential multicurve. 
The DM stacks $\ccM(S,\vec\cP)$ and $\mathrm{P}\ccM(S,\cP)$ are then defined similarly.

As above, we are interested in the complex theory of these objects and we will maintain the same notation for the corresponding complex DM stacks.
A marking $\phi\co S\to C$ then determines the isomorphisms:
\[\begin{array}{ll}
&\G(S,\cP)\cong\pi_1(\cM(S,\cP),[C,\phi(\cP)]),\hspace{0.8cm}\G(S,\vec\cP)\cong\pi_1(\cM(S,\vec\cP),[C,\phi(\cP)])\\
\mbox{and}\hspace{0.5cm}&\\
&\PG(S,\cP)\cong\pi_1(\mathrm{P}\cM(S,\cP),[C,\phi(\cP)]).
\end{array}\]

\subsection{Moduli stacks of hyperelliptic curves of a fixed topological type}\label{HypTopType}
Let $(S,\u)$ be a hyperbolic hyperelliptic surface without boundary. Then, a hyperelliptic complex curve $(C,\iota)$ of topological type $(S,\u)$ is a smooth
complex curve $C$ together with a hyperelliptic involution $\iota\in\Aut(C)$ such that there exists an orientation preserving homeomorphism $f\co C\to S$ 
with the property that $f\circ\iota\circ f^{-1}=\u$.

Hyperelliptic complex curves of topological type $(S,\u)$ are parameterized by a closed smooth irreducible substack $\cM(S,\u)$ of the DM stack 
$\cM(S)$ and an orientation preserving homeomorphism $f\co C\to S$ determines an isomorphism (cf.\ Proposition~1 in \cite{GDH}): 
\[\U(S)\cong\pi_1(\cM(S,\u),[C,\iota]).\]

We let $\mathrm{P}\cM(S,\u)\to\cM(S,\u)$ be the \'etale Galois covering associated to the normal finite index subgroup $\PU(S)$ of $\U(S)$.
In particular, by definition, there is an isomorphism: 
\[\PU(S)\cong\pi_1(\mathrm{P}\cM(S,\u),[C,\iota]).\]

The assignment $(C,\iota)\mapsto(C/\iota,\cB_\iota)$, where $\cB_\iota$ is the branch locus of the natural morphism $C\to C/\iota$, 
defines natural morphisms of DM stacks:
 \begin{equation}\label{gerbes}
 \cM(S,\u)\to\cM(S_{/\u},\cB_\u)\hspace{0.7cm}\mbox{ and }\hspace{0.7cm}\mathrm{P}\cM(S,\u)\to\mathrm{P}\cM(S_{/\u},\cB_\u),
 \end{equation} 
where, as we observed above, there holds $\mathrm{P}\cM(S_{/\u},\cB_\u)\cong\mathrm{P}\cM(S_{/\u}\ssm\cB_\u)$. These morphisms
are, respectively, a $\langle\u\rangle$-gerbe and a $(\langle\u\rangle\cap\PU(S))$-gerbe over their base and induce
on fundamental groups the maps $\U(S)\to\G(S_{/\u},\cB_\u)$ and $\PU(S)\to\PG(S_{/\u},\cB_\u)$ which appear in 
the short exact sequences~(\ref{HypEx}) and~(\ref{PHypEx}), respectively. 
In particular, if $S$ has at least two symmetric punctures, there are natural isomorphisms:
\[\mathrm{P}\cM(S,\u)\cong\mathrm{P}\cM(S_{/\u},\cB_\u)\cong\mathrm{P}\cM(S_{/\u}\ssm\cB_\u).\]

The DM compactification $\ccM(S,\u)$ is the closure of $\cM(S,\u)$ in the proper DM stack $\ccM(S)$. The DM compactification 
$\mathrm{P}\ccM(S,\u)$ is then the normalization of the DM stack $\ccM(S,\u)$ in $\mathrm{P}\cM(S,\u)$.

\subsection{Moduli stacks of marked hyperelliptic curves} 
To a marked hyperelliptic surface  $(S,\u,\cP)$, we can associate the DM stacks $\cM(S,\u,\cP)$, $\cM(S,\u,\vec\cP)$ and $\mathrm{P}\cM(S,\u,\cP)$.
These are just the inverse images of $\cM(S,\u)$ and $\mathrm{P}\cM(S,\u)$ via the natural morphisms $\cM(S,\cP)\to\cM(S)$, $\cM(S,\vec\cP)\to\cM(S)$
and $\mathrm{P}\cM(S,\cP)\to\mathrm{P}\cM(S)$, respectively.

As above, a marking $\phi\co (S,\u)\to (C,\iota)$ determines isomorphisms:
\[\begin{array}{ll}
&\U(S,\cP)\cong\pi_1(\cM(S,\u,\cP),[C,\iota,\phi(\cP)]),\hspace{0.3cm}\U(S,\vec\cP)\cong\pi_1(\cM(S,\u,\vec\cP),[C,\iota,\phi(\cP)])\\
\mbox{and}\hspace{0.5cm}&\\
&\PU(S,\cP)\cong\pi_1(\mathrm{P}\cM(S,\u,\cP),[C,\iota,\phi(\cP)]).
\end{array}\]

Let $Q\in S\ssm\cP$. The Birman short exact sequences~(\ref{BirmanExact}) are then induced by the following smooth families of curves
(with fibers homeomorphic to $S\ssm\cP$):
\[\begin{array}{ll}
&\mathrm{P}\cM(S,\u,\cP\cup\{Q\})\to\mathrm{P}\cM(S,\u,\cP),\hspace{0.4cm}\cM(S,\u,\overrightarrow{\cP\cup\{Q\}})\to\cM(S,\u,\vec\cP)\\
\mbox{and}\hspace{0.5cm}&\\
&\cM(S,\u,\cP,Q)\to\cM(S,\u,\cP).
\end{array}\]
where $\cM(S,\u,\cP,Q)\to\cM(S,\u,\cP\cup\{Q\})$ is the \'etale covering associated to the subgroup $\U(S,\cP,Q)$ of $\U(S,\cP\cup\{Q\})$ 
or, equivalently, $\cM(S,\u,\cP,Q)$ parameterizes curves marked by the set $\cP$ plus a distinguished marked point $Q$.

\subsection{The DM boundary of moduli stacks of stable curves}\label{boundaryDM}
A classical result by Knudsen (cf.\ \cite{Knudsen}) describes the DM boundary of the moduli stack $\CPM(S,\cP)$ associated to a hyperbolic 
marked surface $(S,\cP)$. From our perspective and with our notation, we can formulate this result as follows. 

Let us denote by $S/\s$ the singular topological surface obtained collapsing on $S$ a multicurve, supported on $S\ssm\cP$, in the isotopy class of $\s$. 
We associate to a simplex $\s\in C(S\ssm\cP)$ the closed boundary stratum $\d_\s$ of $\CPM(S,\cP)$ whose generic point parameterizes marked stable 
complex curves of topological type $(S/\s,\cP)$. We then let $\dot{\d}_\s$ be the open substack of $\d_\s$
whose points parameterize stable complex curves of topological type $(S/\s,\cP)$.

Let $S\ssm\s=\coprod_{i=1}^h S^i$ be the decomposition in connected components and let $\cP^i:=\cP\cap S^i$, for $i=1,\ldots,h$. 
There is then a natural finite unramified morphism of DM stacks:
\[b_\s\co\prod_{i=1}^h\CPM(S^i,\cP^i)\to\CPM(S,\cP),\]
with image the closed boundary stratum $\d_\s$ and whose restriction to $\prod_{i=1}^h\PM(S^i,\cP^i)$ is \'etale onto $\dot{\d}_\s$.
Moreover, if we denote by $\td{\d}_\s$ the normalization of $\d_\s$, the morphism $b_s$ factors through a Galois finite \'etale covering:
\[\td{b}_\s\co\prod_{i=1}^h\CPM(S^i,\cP^i)\to\td{\d}_\s,\]
whose Galois group identifies with the image of the stabilizer $\PG(S,\cP)_\s$ in the symmetric group $\Sigma_{\s^\pm}$
(cf.\ the upper exact sequence~(\ref{stabPMCG})).

\subsection{The DM boundary of moduli stacks of hyperelliptic stable curves}\label{boundaryDMHyp}
A similar result holds for moduli stacks of hyperelliptic curves.
As we observed in Section~\ref{complexsym}, for $\s\in C(S,\u,\cP)$, there is a set of disjoint representatives on $S\ssm\cP$, which we also denote by 
$\s$, such that there holds $\u(\s)=\s$ and $(S\ssm\s,u)$ is a, possibly disconnected, hyperelliptic topological surface. We then associate to $\s$ the closed 
boundary stratum $\d_\s^\u$ (resp.\ the open stratum $\dot{\d}_\s^\u$) of $\CPM(S,\u,\cP)$ whose generic point (resp.\ points) parameterizes 
marked stable complex hyperelliptic curves of topological type $(S/\s,\u,\cP)$. 

For $\cP=\emptyset$, we can describe explicitly the strata $\d_\s^\u$ and $\dot{\d}_\s^\u$ in the following way.
As in Section~\ref{disconnected}, let us put the decomposition of $S$ in connected 
components in the form~(\ref{dec}) and let us denote by $\u_i$ the restriction of the hyperelliptic involution $\u$ to $S^i$ for $i=1,\ldots,l$. 
There is then a natural finite unramified morphism of DM stacks:
\begin{equation}\label{DMbHyp}
b_\s^\u\co\prod_{i=1}^l\CPM(S^i,\u_i)\times\prod_{i=l+1}^r\CPM(S^i)\to\CPM(S,\u),
\end{equation}
with image $\d_\s^\u$ and whose restriction to $\prod_{i=1}^l\PM(S^i,\u_i)\times\prod_{i=l+1}^r\PM(S^i)$ is \'etale onto 
$\dot{\d}_\s^\u$. Moreover, if we let $\td{\d}_\s^\u$ be the normalization of $\d_\s^\u$, the morphism $b_s^\u$ factors 
through a finite \'etale Galois covering:
\[\td{b}_\s^\u\co\prod_{i=1}^l\CPM(S^i,\u_i)\times\prod_{i=l+1}^r\CPM(S^i)\to\td{\d}_\s^\u,\]
whose Galois group identifies with the image of the stabilizer $\PU(S)_\s$ in the symmetric group $\Sigma_{\s^\pm}$
(cf.\ the upper exact sequence~(\ref{stabPHyp})).

For $\cP\neq\emptyset$, we can describe explicitly these strata only under the further assumption that $\u$ preserves all connected components
of $S\ssm\s$. With the notations of Section~\ref{boundaryDM}, let us also denote by $\u_i$ the restriction of the hyperelliptic involution $\u$ to $S^i$ 
for $i=1,\ldots,h$. There is then a natural finite unramified morphism of DM stacks:
\[b_\s^\u\co\prod_{i=1}^h\CPM(S^i,\u_i,\cP^i)\to\CPM(S,\u,\cP),\]
with image the closed boundary stratum $\d_\s^\u$ and whose restriction to $\prod_{i=1}^h\PM(S^i,\u_i,\cP^i)$ is \'etale onto $\dot{\d}_\s^\u$.
As above, if we denote by $\td{\d}_\s^\u$ the normalization of $\d_\s^\u$, the morphism $b_s^\u$ factors through a Galois finite \'etale covering:
\[\td{b}_\s^\u\co\prod_{i=1}^h\CPM(S^i,\u_i,\cP^i)\to\td{\d}_\s,\]
whose Galois group identifies with the image of the stabilizer $\PU(S,\cP)_\s$ in the symmetric group $\Sigma_{\s^\pm}$
(cf.\ the exact sequence~(\ref{stabPHyp})).

\subsection{Levels and level structures}
A \emph{level} $\G^\l$ is a normal finite index subgroup of the mapping class group $\G(S,\cP)$. The \emph{level structure} $\cM(S,\cP)^\l$ is the
finite \'etale Galois covering of the moduli stack $\cM(S,\cP)$ associated to this subgroup. 

The \emph{abelian level $\G(S,\cP)^{[m]}$} (or simply $\G^{[m]}$) of order $m$, for an integer 
$m\geq 2$, is the kernel of the natural representation $\rho^{[m]}\co\G(S,\cP)\to\Sp(H_1(\ol{S},\Z/m)$. 
We denote by $\cM(S,\cP)^{[m]}$ the associated \emph{abelian} level structure. 

For $\G^\l$ a level of  $\G(S,\cP)$, the DM compactification $\ccM(S,\cP)^\l$ of the level structure $\cM(S,\cP)^\l$ is defined to be the
normalization of the moduli stack $\ccM(S,\cP)$ in $\cM(S,\cP)^\l$.

Similarly, a \emph{level} $\U^\l$ of the hyperelliptic mapping class group is a normal finite index subgroup of  $\U(S,\cP)$ and the associated finite \'etale 
Galois covering of the moduli stack $\cM(S,\u,\cP)$ is denoted by $\cM(S,\u,\cP)^\l$ and also called a \emph{level structure}.

As above, to an integer $m\geq 2$, we associate the \emph{abelian level} of order $m$ defined 
by $\U(S,\cP)^{[m]}:=\U(S,\cP)\cap\G(S,\cP)^{[m]}$ and the  \emph{abelian} level structure $\cM(S,\u,\cP)^{[m]}$. 
We also let $\PU(S,\cP)^{[m]}:=\PU(S,\cP)\cap\G(S,\cP)^{[m]}$ and $\PU(S,\vec{\cP})^{[m]}:=\PU(S,\vec{\cP})\cap\G(S,\cP)^{[m]}$.

For $\U^\l$ a level of  $\U(S,\cP)$, the DM compactification $\ccM(S,\u,\cP)^\l$ of the level structure $\cM(S,\u,\cP)^\l$ is defined to be the
normalization of the moduli stack $\ccM(S,\u,\cP)$ in $\cM(S,\u,\cP)^\l$.

\subsection{The hyperelliptic Birman exact sequence}\label{HyperellipticBirman} 
Besides the Birman short exact sequences~(\ref{BirmanExact}), which are a consequence of the standard one for the mapping class group,
there are short exact sequences which are peculiar to the hyperelliptic mapping class group. 

For a hyperbolic hyperelliptic surface without boundary $(S,\u)$, let $Q\in S\ssm\cW_\cP(S)$ and $R:=\u(Q)$. Let then $S^\circ:=S\ssm\{Q,R\}$
and label by $Q$ (resp.\ $R$) the puncture obtained by removing $Q$ (resp.\ $R$) from $S$. 
For a marking $(S^\circ,\cP)$, let $\U(S^\circ,\cP)_\circ$ and $\U(S^\circ,\vec\cP)_\circ$ 
be the index $2$ subgroups of $\U(S^\circ,\cP)$ and $\U(S^\circ,\vec\cP)$, respectively,
consisting of those elements which do not swap the two punctures labeled by $Q$ and $R$. We then have:

\begin{theorem}\label{Q}With the above notations, there are natural short exact sequences:
\begin{equation}\label{forgetsympunctures}\begin{array}{ll}
&1\to N_\circ\to\U(S^\circ,\cP)_\circ\to\U(S,\cP,Q)\to 1,\\
\\
&1\to N_\circ\to\U(S^\circ,\vec\cP)_\circ\to\U(S,\ovr{\cP\cup\{Q\}})\to 1\\
\mbox{and}\hspace{0.5cm}&\\
&1\to N_\circ\to\PU(S^\circ,\cP)\to\PU(S,\cP\cup\{Q\})\to 1,
\end{array}
\end{equation}
where $N_\circ$ is the normal subgroup generated by the Dehn twists about symmetric separating simple closed curves on $S^\circ$ bounding 
a disc which contains no marked points and only the two punctures labeled by $Q$ and $R$.
\end{theorem}

\begin{proof}Let $\cM(S^\circ,\u,\cP)_\circ\to\cM(S^\circ,\u,\cP)$ be the \'etale covering associated to the subgroup $\U(S^\circ,\cP)_\circ$ 
of $\U(S^\circ,\cP)$, so that $\cM(S^\circ,\u,\cP)_\circ$ parameterizes marked complex hyperelliptic curves with an order on the set of punctures 
labeled by $Q$ and $R$. 

Let $\cM(S,\u,\cP,Q)^\cW$ be the open substack of $\cM(S,\u,\cP,Q)$
which parameterizes marked complex hyperelliptic curves $[C,\iota,\cP,Q]$ such that the distinguished point labeled by $Q$ is not a 
Weierstrass point of the curve $C$. 

Note that the complement $\partial_\cW:=\cM(S,\u,\cP,Q)\ssm\cM(S,\u,\cP,Q)^\cW$ is a smooth irreducible divisor of $\cM(S,\u,\cP,Q)$ and that
the assignment $[C,\iota,\cP, Q]\mapsto[C\ssm\{Q,\iota(Q)\},\iota,\cP]$
defines a natural isomorphism of moduli stacks: 
\[\cM(S,\u,\cP,Q)^\cW\sr{\sim}{\to}\cM(S^\circ,\u,\cP)_\circ.\]
 
Let $\cM(S,\cP,Q,R)$ be the \'etale covering of $\cM(S,\cP\cup\{Q,R\})$ defined by taking $Q$ and $R$ as distinguished and \emph{ordered}
marked points. The assignment $[C,\iota,\cP, Q]\mapsto[C,\cP,Q,\iota(Q)]$ then defines a natural embedding
$\cM(S,\u,\cP,Q)^\cW\hookra\cM(S,\cP,Q,R)$.
This extends to an embedding $\cM(S,\u,\cP,Q)\hookra\ccM(S,\cP,Q,R)$ which maps the divisor 
$\partial_\cW$ to the irreducible component of the DM boundary of $\ccM(S,\cP,Q,R)$ 
parameterizing stable curves with a rational tail marked by $Q$ and $R$.

Therefore, an element of the fundamental group $\pi_1(\cM(S,\u,\cP,Q)^\cW)$, which can be represented by a small loop
around the divisor $\partial_\cW$, corresponds in the hyperelliptic mapping class group $\U(S^\circ,\u,\cP)_\circ<\G(S^\circ,\cP)_\circ=\G(S,\cP,Q,R)$ 
to a Dehn twist about a symmetric simple closed curve on $S^\circ$ bounding a disc which contains no marked points and only 
the punctures labeled by $Q$ and $R$.

The embedding $\cM(S,\u,\cP,Q)^\cW\subset\cM(S,\u,\cP,Q)$ induces an epimorphism on fundamental groups
with kernel generated by small loops around the divisor $\partial_\cW$. This clearly implies all the claims of the theorem.
\end{proof}

There is a natural epimorphism $p_\circ\co\U(S^\circ,\cP)\to\U(S,\cP)$ induced by filling back with $Q$ and $R$ the two symmetric punctures 
of $S^\circ$. For $\cP=\emptyset$, from the short exact sequence~(\ref{HypEx}) and the standard genus $0$ Birman short exact 
sequence~(\ref{BirmanExact}), it follows that the kernel of this epimorphism is naturally isomorphic to the fundamental group 
$\pi_1(S_{/\u}\ssm\cB_\u,\,p_\u(Q))$, where a simple loop around a puncture of $S_{/\u}\ssm\cB_\u$ is sent to a half twist about a symmetric simple
closed curve of $S^\circ$ bounding a disc which only contains the punctures labeled by $Q$ and $R$.
Therefore, we get the hyperelliptic Birman short exact sequence (cf.\ Theorem~3.2 in \cite{BM} for the case in which $S$ is a closed surface):
\begin{equation}\label{fullHypBirmanEx}
1\to\pi_1(S_{/\u}\ssm\cB_\u,\,p_\u(Q))\to\U(S^\circ)\sr{p_\circ}{\to}\U(S)\to 1.
\end{equation}

Let us then consider the epimorphism $\U(S^\circ,\cP)_\circ\to\U(S,\cP)$ obtained by restriction of $p_\circ$ to the index $2$
subgroup $\U(S^\circ,\cP)_\circ$ of $\U(S^\circ,\cP)$.
This epimorphism factors through the epimorphism $\U(S^\circ,\cP)_\circ\to\U(S,\cP,Q)$ of Theorem~\ref{Q} and
the natural epimorphism $\U(S,\cP,Q)\to\U(S,\cP)$ of the Birman short exact sequence~(\ref{BirmanExact}). 

The epimorphism $\U(S^\circ,\cP)_\circ\to\U(S,\cP)$ is induced by the natural morphism of moduli stacks 
$\cM(S,\u,\cP,Q)^\cW\to\cM(S,\u,\cP)$ (defined by forgetting the marked point $Q$). 

If $\cP=\emptyset$, this morphism is a smooth curve over $\cM(S,\u)$ with fibers homeomorphic to the surface $S\ssm\cW_\cP(S)$. 
Since, as a topological stack, $\cM(S,\u)$ is a $K(\pi,1)$, from the associated long exact sequence of homotopy groups, it then follows:

\begin{theorem}[The hyperelliptic Birman short exact sequence]\label{HypBirman} There are natural short exact sequences:
\begin{equation}\label{HypBirmanEx}\begin{array}{ll}
&1\to\pi_1(S\ssm\cW_\cP(S),Q)\to\U(S^\circ)_\circ\to\U(S)\to 1\\
\mbox{and}\hspace{0.5cm}&\\
&1\to\pi_1(S\ssm\cW_\cP(S),Q)\to\PU(S^\circ)\to\PU(S)\to 1.
\end{array}
\end{equation}

Moreover, the subgroup $N_\circ$, defined in Theorem~\ref{Q}, identifies with the kernel of the natural epimorphism of fundamental groups
$\pi_1(S\ssm\cW_\cP(S),Q)\to\pi_1(S,Q)$. In particular, $N_\circ$ is a free group of infinite rank.
\end{theorem}

\begin{remarks}\label{Bureau}\begin{enumerate}
\item Note that the fundamental group $\pi_1(S\ssm\cW_\cP(S),Q)$ identifies with the index $2$ subgroup of 
$\pi_1(S_{/\u}\ssm\cB_\u,\,p_\u(Q))$ associated to the unramified double covering $p_\u\co S\ssm\cW_\cP(S)\to S_{/\u}\ssm\cB_\u$.
\item As we already observed, there is a natural isomorphism $\PU(S^\circ)\cong\PG(S^\circ_{/\u}\ssm\cB_\u)$. Thus,
for $S=S_g$, we get the isomorphism $\PU(S^\circ)\cong\PG(S_{0,2g+3})$. The natural epimorphism $\PG(S_{0,2g+3})\to\PG(S_{0,2g+2})$, 
induced by filling in a puncture of $S_{0,2g+3}$, then factors through the epimorphism $\PG(S_{0,2g+3})\to\PU(S_g)$ of Theorem~\ref{HypBirman} 
and the epimorphism $\PU(S_g)\to\PG(S_{0,2g+2})$ provided by Birman-Hilden theory.
\end{enumerate}
\end{remarks}


\section{Generators for hyperelliptic mapping class groups}\label{generatorsHMCG}
\subsection{Generators for the pure hyperelliptic mapping class group}\label{genpure}
For genus $\geq 2$, there is a simple generating set for the pure hyperelliptic mapping class group. 

Let us consider first the case $\cP=\emptyset$ and $S$ either a closed surface of genus $\geq 2$ or $S=S_{1,1}$. The following analogue
for hyperelliptic mapping class groups of a result by Humphries for mapping class groups (cf.\ Proposition~2.1 in \cite{Humphries}) then holds:

\begin{proposition}\label{squares}For $(S,\u)$ a hyperelliptic surface of genus $\geq 1$, let $\cQ(S,\u)$ be the subgroup of $\U(S)$ 
generated by squares of nonseparating symmetric Dehn twists. 
\begin{enumerate}
\item For $S$ a closed surface of genus $\geq 2$, we have $\PU(S)=\cQ(S,\u)$.
\item For $S=S_{1,1}$, we have that $\u\notin\cQ(S,\u)$ and $\PU(S)=\langle\u\rangle\cdot\cQ(S,\u)$.
In particular, the natural epimorphism $\PU(S)\to\PG(S_{/\u},\cB_\u)$ induces an isomorphism $\cQ(S,\u)\cong\PG(S_{/\u},\cB_\u)=\PG(S_{0,4})$. 
\end{enumerate}
\end{proposition}

\begin{proof}Let $(\hat{S},\u)$ be a hyperelliptic surface of genus $g=g(S)$ with a single boundary component $\g$,  
fixed by $\u$, and no punctures. The relative hyperelliptic mapping class group $\U(\hat{S},\dd\hat{S})$ is described by the 
Birman-Hilden isomorphism (cf.\ Theorem~9.2 in \cite{FM}):
\[\U(\hat{S},\dd\hat{S})\cong\G(\hat{S}_{/\u},\cB_\u,\dd\hat{S}_{/\u})=\mathrm{B}_{2g+1},\]
where $\mathrm{B}_{2g+1}$ is the braid group in $2g+1$ strands and the isomorphism sends the Dehn twist 
$\tau_\g\in\U(\hat{S},\dd\hat{S})$ to the square of a generator of the center of $\mathrm{B}_{2g+1}$.
\medskip

\noindent (i): By Proposition~\ref{purehyp}, we have $\PU(S)=\U(S)^{[2]}$.
Hence, the abelian level $\U(S)^{[2]}$ and  $\cQ(S,\u)$ both have for image under the natural epimorphism $\U(S)\to\G(S_{/\u},\cB)$
the pure mapping class group $\PG(S_{/\u},\cB)\cong\PG(S_{0,2g+2})$. Therefore, 
$\U(S)^{[2]}=\langle\u\rangle\cdot\cQ(S,\u)$. Hence, in order to prove the first item of the proposition, it is enough to show that $\u\in\cQ(S,\u)$.
We borrow here an argument from the proof of Proposition~2.1 in \cite{Humphries}.

An identity, due to Birman and Hilden (cf.\ Section~4.4 in \cite{BirmanBraids}), implies that $\u\in\cQ(S,\u)$.
More precisely, let $\g_1,\ldots,\g_{2g+1}$ be a maximal chain of symmetric 
nonseparating simple closed curves on $S$. Then, we have:
\[\u=\tau_{\g_1}\tau_{\g_2}\cdot\ldots\cdot\tau_{\g_{2g}}\tau_{\g_{2g+1}}^2\tau_{\g_{2g}}\cdot\ldots\cdot\tau_{\g_2}\tau_{\g_1}.\]
The above identity easily implies that $\u$ is the identity element modulo the normal subgroup $\cQ(S,\u)$ and so that $\u\in\cQ(S,\u)$.
\medskip

\noindent
(ii): For $S=S_{1,1}$, the relative hyperelliptic mapping class group $\U(\hat{S},\dd\hat{S})$ is the universal 
central extension of the mapping class groups $\G(S)$:
\[1\to\tau_\g^\Z\to\U(\hat{S},\dd\hat{S})\to\G(S)\to 1.\]
The Birman-Hilden isomorphism $\U(\hat{S},\dd\hat{S})\cong B_3$ then identifies $\G(S)$ with the quotient of the braid group $B_3$ 
by the square of its center, so that a generator of the center of $B_3$ descends to the hyperelliptic involution $\u$ of $\G(S)$. 

The pure braid group $\mathrm{PB}_n$ contains the center of $\mathrm{B}_n$, splits over it and there is a natural isomorphism 
$\mathrm{PB}_n\cong\Z\times\PG(S_{0,n+1})$ (cf.\ Section~9.3 in \cite{FM}). 
By means of the Birman-Hilden isomorphism above, the pure hyperelliptic mapping class group $\PU(S)$ is then identified with 
the quotient of $\mathrm{PB}_3$ by the square of its center and so there is a natural isomorphism 
$\PU(S)\cong\Z/2\times\PG(S_{0,4})$ which sends the hyperelliptic involution to the factor $\Z/2$.

The Dehn twists about essential simple closed curves in $S_{0,4}$ naturally lift to squares of nonseparating Dehn twists in $\PU(S)$ 
and these generate a subgroup $\cQ(S,\u)$ isomorphic to $\PG(S_{0,4})$. This completes the proof of the second item of the proposition.
\end{proof}

For the general case, we need to introduce some notation. We say that a Dehn twist about a symmetric, resp.\ nonseparating, resp.\ 
separating simple closed curve on $S$ is \emph{symmetric}, resp.\ \emph{nonseparating}, resp.\ \emph{separating}.
A \emph{symmetric bitwist} is the product of two Dehn twists about a symmetric pair of simple closed curves.
A \emph{symmetric bounding pair map} is a bounding pair map about a symmetric pair of simple closed curves. We then have:

\begin{theorem}\label{generationpure} Let $(S,\u,\cP)$ be a marked hyperelliptic surface:
\begin{enumerate} 
\item For $g(S)\geq 2$, the pure hyperelliptic mapping class group $\PU(S,\cP)$ is generated by squares of nonseparating symmetric 
Dehn twists, by Dehn twists about symmetric separating simple closed curves bounding unmarked discs (with punctures)
and by symmetric bitwists about pairs of simple closed curves which bound unpunctured annuli. 
\item For $g(S)=1$ and $\sharp\cW_\circ(S)\geq 1$,
the group $\PU(S,\cP)$ is generated by the elements mentioned above together with a hyperelliptic involution 
$\td\u\in\PU(S,\cP)$ which lifts the hyperelliptic involution $\u\in\PU(S)$.
\end{enumerate}
\end{theorem}

\begin{proof}Since there is a natural isomorphism $\PU(S,\cP)\cong\PU(\ring{S},\cP)$, 
it is not restrictive to assume that $\partial S=\emptyset$.
For a similar reason, we may also assume that, for $g(S)\geq 2$, we have $\cW_\circ(S)=\emptyset$ and, for $g(S)=1$, we have $\sharp\cW_\circ(S)=1$. 

If $S$ has $n$ pairs of symmetric punctures, by Theorem~\ref{Q} and a simple induction, there is a natural epimorphism 
$\PU(S,\cP)\to\PU(\ol{S},\cP\cup\{Q_1,\ldots,Q_n\})$
whose kernel is generated by Dehn twists about symmetric separating simple closed curves on $S$ bounding 
discs which contain no marked points and at least a pair of symmetric punctures. Thus, we are reduced to prove 
Theorem~\ref{generationpure} for $S$ a closed surface of genus $\geq 2$ and for $S=S_{1,1}$.

A bounding pair map $\tau_{\g_0}\tau_{\g_1}^{-1}$ about a symmetric pair $\g_0,\g_1$ of simple closed curves which bound a marked but 
unpunctured annulus is the product of the symmetric bitwist $\tau_{\g_0}\tau_{\g_1}$ with the inverse of $\tau_{\g_1}^{2}$ and these
are both part of our generating set. Since the kernel of the natural epimorphism $\PG(S,\cP)\to\PG(S)$ is generated 
by bounding pair maps about symmetric pairs of simple closed curves 
which  bound a marked unpunctured annulus, the theorem is reduced to the case $\cP=\emptyset$.
The conclusion then follows from Proposition~\ref{squares}.
\end{proof}

\begin{corollary}\label{simplyconnected}Let $(S,\u,\cP)$ be a marked hyperelliptic surface:
\begin{enumerate}
\item For $g(S)\geq 2$, there holds $\pi_1(\CPM(S,\u,\cP))=\{1\}$.
\item For $g(S)=1$ and $\sharp\cW_\circ(S)\geq 1$, there holds $\pi_1(\CPM(S,\u,\cP))\cong\Z/2$.
\end{enumerate}
\end{corollary}

\begin{proof}This follows from the geometric interpretation of the generators of the fundamental group of $\mathrm{P}\cM(S,\u,\cP)$
as loops about the irreducible components of its DM boundary and Theorem~\ref{generationpure}.
\end{proof}

We can determine the gerbe structure of the moduli stacks considered above:

\begin{proposition}\label{gerbe}Let $(S,\u)$ be a hyperelliptic surface:
\begin{enumerate}
\item For $S=S_{1,1}$, the DM compactification $\ccM(S)^{[2]}$ of the abelian level structure $\cM(S)^{[2]}$ is a trivial 
$\Z/2$-gerbe over the moduli stack $\CPM(S_{/\u},\cB_\u)$.
\item For $S$ closed of genus $\geq 2$, the abelian level structure $\cM(S,\u)^{[2]}$ is a trivial 
$\Z/2$-gerbe over the moduli stack $\PM(S_{/\u},\cB_\u)$.
\item For $S$ closed of genus $\geq 2$, the DM compactification $\ccM(S,\u)^{[2]}$ of $\cM(S,\u)^{[2]}$ is \emph{not} a trivial 
$\Z/2$-gerbe over the moduli stack $\CPM(S_{/\u},\cB_\u)$.
\end{enumerate}
\end{proposition}

\begin{proof}(i): This is well known (see, for instance, item (h) of Theorem~1.4 in \cite{GT}) but also an almost immediate consequence of item (i) 
of Proposition~\ref{squares}.
\medskip

\noindent
(ii): Assigning to a smooth projective genus $0$ curve $C$ endowed with an ordered set of marked points $\cB_\u$, the hyperelliptic curve $C'$ such
that $C'/\langle\u\rangle\cong C$ with branch locus $\cB_\u$, defines a morphism $\bar s\co\PM(S_{/\u},\cB_\u)\to\cM(S,\u)$ 
(cf.\ item (h) of Theorem~1.4 in \cite{GT}), which, composed with the
natural morphism $\cM(S,\u)\to\cM(S_{/\u},\cB_\u)$, is just the natural morphism $\PM(S_{/\u},\cB_\u)\to\cM(S_{/\u},\cB_\u)$.

Since $\cM(S)^{[2]}\cong\cM(S,\u)\times_{\cM(S_{/\u},\cB_\u)}\PM(S_{/\u},\cB_\u)$, the morphism $\bar s$ lifts to an \'etale morphism
$s\co\PM(S_{/\u},\cB_\u)\to\cM(S)^{[2]}$ which is a section of $\cM(S,\u)^{[2]}\to\PM(S_{/\u},\cB_\u)$ and so trivializes this $\Z/2$-gerbe.
\medskip

\noindent
(iii): Let us assume, by contradiction, that the gerbe $\ccM(S,\u)^{[2]}\to\CPM(S_{/\u},\cB_\u)$ is trivial.
A splitting of this gerbe is a (representable) morphism $\iota\co\CPM(S_{/\u},\cB_\u)\to\ccM(S,\u)^{[2]}$ which composed with the latter is the identity
and so is \'etale. By (i) of Corollary~\ref{simplyconnected}, the stack $\ccM(S,\u)^{[2]}$ is simply connected which implies that $\bar\iota$ and then 
$\iota$ is an isomorphism which gives a contradiction.
\end{proof}





\subsection{Generators for the full hyperelliptic mapping class group}\label{genfull}
We have the following generalizations of a well known result by Birman and Hilden (cf.\ \cite{BH}):

\begin{proposition}\label{generationfull}Let $(S,\u,\cP)$ be a marked hyperelliptic surface of genus $\geq 1$ such that $\cW_\circ(S)=\cW_\dd(S)=\emptyset$. 
Then the hyperelliptic mapping class group $\U(S,\cP)$ is generated by nonseparating symmetric Dehn twists and by half twists about simple closed curves 
which bound a disc containing only: either a pair of marked points, a pair of symmetric punctures or a pair of symmetric boundary components.
\end{proposition}

\begin{proof}


The kernel of the natural epimorphism $\U(S,\cP)\to\U(S)$ identifies with the fundamental group of the configuration space
of $\sharp\cP$ points on the surface $S$ and it is well known that this group is generated by bounding pair maps and half twists about simple 
closed curves which bound a disc containing a pair of marked points. We are then reduced to prove the proposition for the
hyperelliptic mapping class group $\U(S)$.

It is easy to reduce to the case when $\partial S=\emptyset$.
By the hyperelliptic Birman short exact sequence~(\ref{fullHypBirmanEx}), we can then reduce further to the case when $S$ is a closed surface, 
that is to say to the case treated by Birman and Hilden.
\end{proof}

\begin{proposition}\label{generators}Let $(S,\u,\cP)$ be either a marked hyperelliptic closed surface of genus $\geq 2$ or a marked $1$-punctured 
surface of genus $1$. Then, the hyperelliptic mapping class group $\U(S,\vec\cP)$ is generated by nonseparating symmetric Dehn twists.
\end{proposition}

\begin{proof}This is essentially the same proof as that of Proposition~\ref{generationfull}. We only need to observe that the kernel of the natural 
epimorphism $\U(S,\vec\cP)\to\U(S)$ identifies with the fundamental group of the configuration space of $\sharp\cP$ \emph{ordered} points on 
the surface $S$. The latter, for $g(S)\geq 1$, is generated by bounding pair maps which are products of powers of nonseparating symmetric Dehn twists.
\end{proof}


\section{Homotopy type of the complex of symmetric nonseparating curves}\label{homtype}
Let $(S,\u,\cP)$ be a marked hyperelliptic surface without boundary and Weierstrass punctures. 
In Section~\ref{complexsym}, we defined a version of the nonseparating curve complex for the hyperelliptic mapping class group $\U(S,\cP)$ of $(S,\u,\cP)$. 
Let us recall that the latter, denoted by $C_\mathrm{ns}(S\ssm\cP)$, is the simplicial complex whose $k$-simplices consist of unordered 
$(k+1)$-tuples $\{\g_0,\ldots,\g_k\}$ of distinct isotopy classes of pairwise disjoint simple closed curves on $S\ssm\cP$ such that 
$S\ssm\{\g_0,\ldots,\g_k\}$ is connected and that the complex of symmetric nonseparating curves $C_\mathrm{ns}(S,\u,\cP)$ is then 
defined to be the full subcomplex of $C_\mathrm{ns}(S\ssm\cP)$ generated by isotopy classes of symmetric nonseparating simple closed curves. 

Let us observe that filling in a puncture induces surjective maps $C_\mathrm{ns}(S\ssm(\cP\cup\{Q\}))\to C_\mathrm{ns}(S\ssm\cP)$ and 
$C_\mathrm{ns}(S,\u,\cP\cup\{Q\})\to C_\mathrm{ns}(S,\u,\cP)$. Both maps, or, better, their geometric realizations, 
are homotopy equivalences. For the first map, this follows from Harer's results (cf.\ Theorem~1.1 in \cite{Harer}) but we want to give here a more direct 
proof which can be easily adapted to the hyperelliptic case:

\begin{proposition}\label{homotopyequiv1}There is a natural surjective map $C_\mathrm{ns}(S_{g,n+1})\to C_\mathrm{ns}(S_{g,n})$ induced by filling in 
the last puncture, which is a homotopy equivalence, in the sense that the geometric realization of this map is a homotopy equivalence. In particular, 
the homotopy type of $C_\mathrm{ns}(S_{g,n})$ is independent of $n$. 
\end{proposition}

\begin{proof}Let $\cT(S_{g,n})$ be the Teichm\"uller space associated to the surface $S_{g,n}$ and let $\wT(S_{g,n})$ be the 
\emph{Harvey cuspidal bordification} of $\cT(S_{g,n})$ (cf.\  \cite{Harvey}) which can be explicitly
described as follows. Let $\wh{\cM}(S_{g,n})$ be the real oriented blow-up of the DM compactification $\ccM(S_{g,n})$ of the stack 
$\cM(S_{g,n})$ along its DM boundary. This is a real analytic DM stack with corners whose boundary $\dd\wM(S_{g,n}):=\wM(S_{g,n})\ssm\cM(S_{g,n})$ 
is homotopic to a deleted tubular neighborhood of the DM boundary of $\ccM(S_{g,n})$. For an intrinsic construction of $\wh{\cM}(S_{g,n})$, we
apply Weil restriction of scalars $\Res_{\C/\R}$ to the complex DM stack $\ccM(S_{g,n})$ and blow up $\Res_{\C/\R}\ccM(S_{g,n})$ along the codimension
two substack $\Res_{\C/\R}\dd\ccM(S_{g,n})$. The real oriented blow-up $\wh{\cM}(S_{g,n})$ is obtained taking the set of real points of this blow-up 
and cutting along the exceptional divisor (which has codimension $1$).
The natural projection $\wh{\cM}(S_{g,n})\to\ccM(S_{g,n})$ restricts over each codimension $k$ stratum to a bundle in $k$-dimensional tori. 

The Harvey cuspidal bordification $\wT(S_{g,n})$ is then the universal cover of the real analytic stack $\wh{\cM}(S_{g,n})$. 
It can be shown that $\wT(S_{g,n})$ is representable and thus a real analytic manifold with corners (cf.\ Section~4 in \cite{sym}).
The inclusion $\cM(S_{g,n})\hookra\wh{\cM}(S_{g,n})$  is a homotopy equivalence and then induces an inclusion of the respective
universal covers $\cT(S_{g,n})\hookra\wT(S_{g,n})$, which is also a homotopy equivalence. The \emph{ideal boundary} of the Teichm{\"u}ller space 
$\cT(S_{g,n})$is defined to be $\dd\,\wT(S_{g,n}):=\wT(S_{g,n})\ssm\cT(S_{g,n})$.
 
The Harvey cuspidal bordification $\wT(S_{g,n})$ is endowed with a natural action of
the mapping class group $\G(S_{g,n})$ and has the property that the nerve of the cover of the ideal boundary $\partial\,\wT(S_{g,n})$ 
by (analytically) irreducible components is described by the curve complex $C(S_{g,n})$ (cf.\   \cite{Harvey} or \cite{sym} for more details).
Since such irreducible components are all contractible, there is a $\G(S_{g,n})$-equivariant weak homotopy equivalence between 
$\partial\,\wT(S_{g,n})$ and the geometric realization of $C(S_{g,n})$ (cf.\ Theorem~2 in \cite{Harvey}). 

Let $\dd_\mathrm{ns}\wT(S_{g,n})$ be the locus of the ideal boundary which parameterizes irreducible nodal Riemann surfaces, that is to say
$\dd_\mathrm{ns}\wT(S_{g,n})$ is the inverse image in $\wT(S_{g,n})$, via the natural map $\wT(S_{g,n})\to\ccM(S_{g,n})$, 
of the locus in $\dd\ccM(S_{g,n})$ which parameterizes irreducible nodal complex algebraic curves.
Then, the nonseparating curve complex $C_\mathrm{ns}(S_{g,n})$ naturally identifies with the nerve of the covering 
of $\dd_\mathrm{ns}\wT(S_{g,n})$ by irreducible components, which are also contractible. 

The natural map $\wh{\cM}(S_{g,n+1})\to\wh{\cM}(S_{g,n})$, induced by filling in the last puncture, lifts to a map of universal covers
$\wT(S_{g,n+1})\to\wT(S_{g,n})$ whose restriction $\dd_\mathrm{ns}\wT(S_{g,n+1})\to\dd_\mathrm{ns}\wT(S_{g,n})$
is homotopy equivalent to the geometric realization of the map $C_\mathrm{ns}(S_{g,n+1})\to C_\mathrm{ns}(S_{g,n})$.

The fibers of the map $\wh{\cM}(S_{g,n+1})\to\wh{\cM}(S_{g,n})$ are homeomorphic to the compact surface with boundary $\wh{S}_{g,n}$ obtained from
$S_{g,n}$ replacing punctures with circles. The boundary of a fiber of the map $\wh{\cM}(S_{g,n+1})\to\wh{\cM}(S_{g,n})$ is in $\dd\wh{\cM}(S_{g,n+1})$ 
and parameterizes nodal \emph{reducible} curves. It follows that the map $\dd_\mathrm{ns}\wT(S_{g,n+1})\to\dd_\mathrm{ns}\wT(S_{g,n})$
is a Serre fibration with fibers consisting of Poincar\'e discs (without boundary) and, in particular, is a homotopy equivalence. Hence, the proposition follows.
\end{proof}

\begin{remark}\label{homotopyequiv2}Let $C(S_{g,n})$ (resp.\ $C_\mathrm{nb}(S_{g,n})$) be the simplicial complex whose $k$-simplices consist 
of unordered $(k+1)$-tuples $\{\g_0,\ldots,\g_k\}$ of distinct isotopy classes of pairwise disjoint nonperipheral simple closed curves on $S_{g,n}$
(resp.\ and which, moreover, do not bound a $2$-punctured disc). A statement similar to Proposition~\ref{homotopyequiv1} then also holds for the 
simplicial sets $C_\mathrm{nb}(S_{g,n+1})_\bullet$ and $C(S_{g,n})_\bullet$ associated, respectively, to $C_\mathrm{nb}(S_{g,n+1})$ and $C(S_{g,n})$. 
There is still indeed a natural map $C_\mathrm{nb}(S_{g,n+1})_\bullet\to C(S_{g,n})_\bullet$, 
induced by filling in the last puncture (but observe that this is not anymore true, for $n\geq 1$, if we replace $C_\mathrm{nb}(S_{g,n+1})$ with the 
full curve complex $C(S_{g,n+1})$) and essentially the same argument of Proposition~\ref{homotopyequiv1} applies. 
In particular, this implies the well-known fact that $C(S_{g,1})$ and $C(S_{g})$ have the same homotopy type.
\end{remark}

\begin{proposition}\label{homotopyequiv}The geometric realization of the map $C_\mathrm{ns}(S,\u,\cP\cup\{Q\})\to C_\mathrm{ns}(S,\u,\cP)$,
induced forgetting the point $Q$, is a homotopy equivalence. In particular, the homotopy type of $C_\mathrm{ns}(S,\u,\cP)$ is independent of 
$\sharp\,\cP$. 
\end{proposition}

\begin{proof}The fixed point set $\cT(S)^\u$ of the hyperelliptic involution $\u$ is contractible and identifies with the universal cover
of the moduli stack of hyperelliptic curves $\cM(S,\u)$ and the fixed point set $\wT(S)^\u$ is a bordification of $\cT(S)^\u$.
For $n\geq 1$, we then let $\cT(S,\cP)^\u$ and $\wT(S,\cP)^\u$ be, respectively, the inverse images of $\cT(S)^\u$ and 
$\wT(S)^\u$ via the natural map $\wT(S,\cP)\to\wT(S)$. Clearly, $\wT(S,\cP)^\u$ is a bordification of $\cT(S,\cP)^\u$ for $n\geq 1$ as well. 

Let then $\partial\,\wT(S,\cP)^\u:=\wT(S,\cP)^\u\ssm\cT(S,\cP)^\u$ and
let $\dd_\mathrm{ns}\wT(S,\cP)^\u$ be the portion of $\partial\,\wT(S,\cP)^\u$ which parameterizes irreducible nodal 
Riemann surfaces. The hyperelliptic nonseparating curve complex $C_\mathrm{ns}(S,\u,\cP)$ naturally identifies with the nerve of the covering 
of $\dd_\mathrm{ns}\wT(S,\cP)^\u$ by irreducible components, which are contractible. 
Therefore, the geometric realization of the natural map $C_\mathrm{ns}(S,\u,\cP\cup\{Q\})\to C_\mathrm{ns}(S,\u,\cP)$ is homotopy equivalent to 
the natural map $\dd_\mathrm{ns}\wT(S,\cP\cup\{Q\})^\u\to\dd_\mathrm{ns}\wT(S,\cP)^\u$. In the proof of 
Proposition~\ref{homotopyequiv1}, we saw that the latter map is a Serre fibration with fibers consisting of Poincar\'e discs and,
in particular, is a homotopy equivalence. Hence, the proposition follows.
\end{proof}

With the notations of Section~\ref{HyperellipticBirman}, there is also a natural map $C_\mathrm{ns}(S^\circ,\u)\to C_\mathrm{ns}(S,\u)$.
As we observed there, the corresponding morphism $\cM(S^\circ,\u)_\circ\to\cM(S,\u)$ is a smooth curve with fibers homeomorphic to
the surface $S\ssm\cW_\cP(S)$. By the same argument of Proposition~\ref{homotopyequiv}, we then have:

\begin{proposition}\label{homotopyequivbis}The geometric realization of the map $C_\mathrm{ns}(S^\circ,\u)\to C_\mathrm{ns}(S,\u)$,
induced filling in the punctures $Q$ and $R$, is a homotopy equivalence. 
\end{proposition}

For $n\geq 4$, let us define the full sub-complex $C_b(S_{0,n})$ of the complex of curves $C(S_{0,n})$ generated by the 
vertices corresponding to simple closed curves on $S_{0,n}$ which bound a $2$-punctured disc. 

\begin{theorem}\label{sym complex}For $n\geq 5$, the curve complex $C_b(S_{0,n})$ is of dimension 
$[n/2]-1$ and, for $n\geq 4$, is $([(n+1)/2]-3)$-connected (here $[\_]$ denotes the integral part). 
\end{theorem}

\begin{proof}The $0$-connectivity of the curve complex $C_b(S_{0,n})$, for $n\geq 5$, can be proved by the same 
argument which proves the connectivity of standard curve complexes (see, for instance, Section~4.1.1 and 4.1.2 
in \cite{FM}).

The $([(n+1)/2]-3)$-connectivity of $C_b(S_{0,n})$, for $n\geq 4$, is derived from the fact that $C(S_{0,n})$, for $n\geq 4$, is 
$(n-5)$-connected (cf.\ Theorem~1.2 in \cite{Harer}) and the same argument used in the proof of Theorem~1.1 \cite{Harer}.
Let us observe first that the curve complex $C_b(S_{0,n})$, for $n=4,5$, coincides with $C(S_{0,n})$ and thus is 
$(n-5)$-connected. The case $n=6$ of the theorem is then also clear.  Let us proceed by induction on $n$. 
So let us assume that the theorem holds for all $k$ such that $6\leq k<n$ and let us prove it for $k=n$.

We almost follow word by word Harer's argument in the proof of Theorem~1.1 \cite{Harer}. Let $S^m$ be an
$m$-dimensional piecewise linear simplicial sphere, for $m<[(n+1)/2]-3$, and $f\co S^m\to C_b(S_{0,n})\subseteq C(S_{0,n})$ 
be a simplicial map. Since $C(S_{0,n})$ is $(n-5)$-connected and $m<n-5$, there is an $(m+1)$-dimensional piecewise linear 
simplicial disc $B$ such that $\dd B=S^m$ and a simplicial map $\hat f\co B\to C(S_{0,n})$ which extends $f$.

We say that a simplex $\s$ of $B$ is \emph{pure} if $\bar f(\s)=\{\g_0,\ldots,\g_s\}$, where no $\g_i$ bounds a
disc with two punctures, and define the \emph{complexity} $c(\bar f)$ of $\bar f$ to be the largest $k$ for which
some $k$-simplex of $B$ is pure.

Let then $\s$ be a pure $c$-simplex of $B$ with $c:=c(\bar f)$, and let $\bar f(\s)=\{\g_0,\ldots,\g_h\}$, with
$h\leq c\leq [n+1/2]-2$. Let $S_1,\ldots, S_{h+1}$ be the connected components of $S_{0,n}\ssm\{\g_0,\ldots,\g_h\}$.
The join $\Delta_\s=C_b(S_1)\ast\ldots\ast C_b(S_{h+1})$ identifies with a subcomplex of $C_b(S_{0,n})$ and, by
hypothesis, $\bar f(\mathrm{Link}_B(\s))$ lies in $\Delta_\s$. Moreover, if $S_i$ has $n_i$ punctures,
there holds $\sum_{i=1}^{h+1} n_i=n+2h+2$. 

By the induction hypotheses, the connectivity of $\Delta_\s$ is described by the following formula. The join
$\Delta_\s$ is $\mu$-connected, where: 
$$\begin{array}{ll}
\mu:&=\sum_{i=1}^{h+1}\left(\left[\dfrac{n_i+1}{2}\right]-3\right)+h+1=\sum_{i=1}^{h+1}\left[\dfrac{n_i+1}{2}\right]-2h-2\geq\\&\\
&\geq\left[\dfrac{n+2h+2}{2}\right]-2h-2=\left[\dfrac{n}{2}\right]-h-1\geq\left[\dfrac{n+1}{2}\right]-2-h\geq\left[\dfrac{n+1}{2}\right]-2-c.
\end{array}$$
Then, the rest of the argument proceeds exactly as in the last two paragraphs of the proof of Theorem~1.1 in \cite{Harer}.
\end{proof}

The quotient map $S_{g}\to S_{g}/\u=S_{0,2g+2}$ identifies $C_\mathrm{ns}(S_{g},\u)$ with the curve complex $C_b(S_{0,2g+2})$. 
By Proposition~\ref{homotopyequiv}, Proposition~\ref{homotopyequivbis} and Theorem~\ref{sym complex}, we then have:

\begin{corollary}\label{hyp complex}For $g(S)\geq 2$, the hyperelliptic nonseparating curve complex 
$C_\mathrm{ns}(S,\u,\cP)$  is spherical of dimension $g-1$. 
\end{corollary}

In Section~\ref{representations}, we will need the following result:

\begin{proposition}\label{Farey}Let $\alpha$ and $\beta$ be nonseparating symmetric simple closed curves on a marked hyperelliptic surface $(S,\u,\cP)$, 
with $g(S)\geq 2$. There is then a chain of symmetric simple closed curves $\g_0,\ldots,\g_k$ on $S\ssm\cP$ such that 
$\g_0=\alpha$, $\g_k=\beta$ and $\g_i$ intersects $\g_{i+1}$ transversally in a single point, for $i=0,\ldots,k-1$.
\end{proposition}

\begin{proof}It is easy to see that, for two disjoint nonseparating symmetric simple closed curves $\alpha$ and $\beta$ on 
$S\ssm\cP$, there is a a symmetric simple closed curve $\g$ which intersects both $\alpha$ and $\beta$ transversally in a single point. The claim of the 
proposition then follows from the connectedness of the curve complex $C_\mathrm{ns}(S,\u,\cP)$ (cf.\ Corollary~\ref{hyp complex} above).
\end{proof}

\section{Linear representations of hyperelliptic mapping class groups}\label{representations}

\subsection{The Putman-Wieland conjecture}\label{PWconjecture}
In \cite{PW}, Putman and Wieland made a conjecture about virtual linear representations of mapping class groups 
which we present here in a slightly different but equivalent formulation. Let $(S,\cP)$ be a hyperbolic marked closed surface 
and let $Q\in S\ssm\cP$. We then have the Birman exact sequence:
\[1\to\pi_1(S\ssm\cP,Q)\to\G(S,\cP,Q)\sr{p_Q}{\to}\G(S,\cP)\to 1.\]
For a finite index subgroup $\G^\l$ of $\G(S,\cP,Q)$, let $\Pi^\l:=\pi_1(S\ssm\cP,Q)\cap\U$.
By restriction, we then get a short exact sequence:
\[1\to\Pi^\l\to\G^\l\to p_Q(\G^{\l})\to 1.\]

Let $(S\ssm\cP)^\l\to S\ssm\cP$ be the unramified covering associated to the subgroup $\Pi^\l$ of $\pi_1(S\ssm\cP,Q)$ and let $S^\l\to S$ 
be the (possibly ramified) covering obtained from this filling in all punctures. The outer representation associated to the above short exact
sequence induces a linear representation: $p_Q(\G^{\l})\to\Sp(H^1(S^\l,\Q))$. 
The \emph{Putman-Wieland conjecture} states that, for $g(S)\geq 2$, all nontrivial orbits of this representation are infinite.

\subsection{The hyperelliptic Putman-Wieland problem}\label{HypProblem}
From now to the end of Section~\ref{representations}, we assume that $(S,\u,\cP)$ is a hyperbolic marked hyperelliptic \emph{closed} surface.
Let us consider the Birman exact sequence~(\ref{BirmanExact}):
\[1\to\pi_1(S\ssm\cP,Q)\to\U(S,\cP,Q)\sr{f_Q}{\to}\U(S,\cP)\to 1.\]
For a finite index subgroup $\U^\l$ of $\U(S,\cP,Q)$, let $\Pi^\l:=\pi_1(S\ssm\cP,Q)\cap\U^\l$. 
There is then a short exact sequence:
\begin{equation}\label{inducedshort}
1\to\Pi^\l\to\U^\l\to f_Q(\U^\l)\to 1
\end{equation}
and an associated representation $\rho^\l\co  f_Q(\U^\l)\to\Out(\Pi^\l)$. 

Let $(S\ssm\cP)^\l\to S\ssm\cP$ be the unramified covering associated to the subgroup $\Pi^\l$ of $\pi_1(S\ssm\cP,Q)$ and let $S^\l\to S$ 
be the (possibly ramified) covering obtained from this filling in all punctures. The representation $\rho^\l$ then induces a linear representation:
\begin{equation}\label{virtualhyp}
L\rho^\l\co  f_Q(\U^\l)\to\Sp(H^1(S^\l,\Q)).
\end{equation}

The \emph{hyperelliptic Putman-Wieland problem} asks whether, for every finite index subgroup 
$\U^\l$ of $\U(S,\cP,Q)$, all nontrivial orbits of $L\rho^\l$ are infinite. 

In Corollary~\ref{NoHypPutWie}, we will show that the answer is no for $g(S)\geq 2$ and $\sharp\cP\geq 4$. In particular, since for $g(S)=2$, 
there holds $\U(S,\cP,Q)=\G(S,\cP,Q)$, this will also provide a counterexample to Putman-Wieland conjecture in genus $2$.



\subsection{Geometric interpretation of the hyperelliptic Putman-Wieland problem}
The hyperelliptic Putman-Wieland problem can be reformulated as follows.

Let $\cM(S,\u,\cP,Q)^\l\to\cM(S,\u,\cP,Q)$ be the finite \'etale covering associated to the subgroup $\U^\l$ of $\U(S,\u,\cP,Q)$. 
For simplicity, let us denote by $\cM(S,\u,\cP)^{f_Q(\l)}\to\cM(S,\u,\cP)$ the finite \'etale covering associated to the subgroup $f_Q(\U^\l)$ of $\U(S,\u,\cP)$. 
This coincides with the normalization of $\cM(S,\u,\cP)$ in $\cM(S,\u,\cP,Q)^\l$. The natural morphism $\phi_Q\co\cM(S,\u,\cP,Q)^\l\to\cM(S,\u,\cP)^{f_Q(\l)}$ 
is then a connected smooth curve with fibers homeomorphic to the surface $(S\ssm\cP)^\l$.

Extending the morphism $\phi_Q$ to DM compactifications, we obtain a projective flat morphism 
$\bar\phi_Q\co\ccM(S,\u,\cP,Q)^\l\to\ccM(S,\u,\cP)^{f_Q(\l)}$, with fibers semistable curves, such that
its restriction over $\cM(S,\u,\cP)^{f_Q(\l)}$: 
\begin{equation}\label{relativecomp}
\td{\phi}_Q\co\cC^\l\to\cM(S,\u,\cP)^{f_Q(\l)}
\end{equation}
is a smooth proper curve with fibers homeomorphic to the surface $S^\l$ defined above.
The representation~(\ref{virtualhyp}) $L\rho^\l\co  f_Q(\U^\l)\to\Sp(H^1(S^\l,\Q))$ is then the monodromy representation associated 
to the proper smooth curve~(\ref{relativecomp}).

This interpretation is important because it allows to implement Hodge theoretic techniques. Indeed, since moduli stacks of (hyperelliptic) 
curves are $K(\pi,1)$-spaces, the homology and cohomology of (hyperelliptic) mapping class groups with rational coefficients identify with 
the singular homology and cohomology of smooth DM stacks, as such they are endowed with a natural mixed Hodge structure. 
The same holds for finite index subgroups of (hyperelliptic) mapping class groups. A first interesting consequence is the following lemma:

\begin{lemma}\label{lemma:AH} Let $(S,\u,\cP)$ be a hyperbolic marked hyperelliptic closed surface and $\U^\l$ 
a finite index subgroup of $\PU(S,\cP\cup\{Q\})$. Associated to the short exact sequence~(\ref{inducedshort}), 
there is a short exact sequence of polarized Hodge structures of weight one:
\begin{equation}\label{purexact}
0\to W_1H^1(f_Q(\U^\l),\Q)\to  W_1H^1(\U^\l,\Q)\to H^1(S^\l,\Q))^{f_Q(\U^\l)}\to 0.
\end{equation}
\end{lemma}

\begin{proof}By a theorem of Deligne (cf.\ \cite{deligne:lefschetz}), the associated Leray spectral sequence associated to the proper curve 
$\td{\phi}_Q\co\cC^\l\to\cM(S,\u,\cP)^{f_Q(\l)}$ degenerates at the $E_2$ level and there is a short exact sequence:
\[0\to H^1(\cM(S,\u,\cP)^{f_Q(\l)},\Q)\sr{f_\ast}{\to}H^1(\cC^\l,\Q)\to H^0(\cM(S,\u,\cP)^{f_Q(\l)},R^1\td{\phi}_{Q\ast}\Q)\to 0,\]
where the space $H^0(\cM(S,\u,\cP)^{f_Q(\l)},R^1\td{\phi}_{Q\ast}\Q)$ identifies with the space of monodromy invariants
$H^1(\td{\phi}_Q^{-1}(z),\Q)^{f_Q(\U^\l)}$, for all $z\in\cM(S,\u,\cP)^{f_Q(\l)}$.

It follows that, after identifying $H^1(\td{\phi}_Q^{-1}(z),\Q)^{f_Q(\U^\l)}$ with $H^1(S^\l,\Q))^{f_Q(\U^\l)}$,
there is a short exact sequence:
\[0\to H^1(\cM(S,\u,\cP)^{f_Q(\l)},\Q)\sr{\td{\phi}_Q^\ast}{\to}H^1(\cC^\l,\Q)\to  H^1(S^\l,\Q))^{f_Q(\U^\l)}\to 0.\]
This sequence remains exact passing to its weight $1$ part:
\[0\to W_1 H^1(\cM(S,\u,\cP)^{f_Q(\l)},\Q)\sr{\td{\phi}_Q^\ast}{\to}W_1 H^1(\cC^\l,\Q)\to  H^1(S^\l,\Q))^{f_Q(\U^\l)}\to 0.\]
The conclusion of the lemma then follows observing that we have:
\[\begin{array}{ll}
&H^1(\cM(S,\u,\cP)^{f_Q(\l)},\Q)=H^1(f_Q(\U^\l),\Q)\\
\mbox{and}\hspace{0.5cm}&\\
&W_1 H^1(\cC^\l,\Q)=W_1 H^1(\cM(S,\u,\cP,Q)^\l,\Q)=W_1 H^1(\U^\l,\Q).
\end{array}\]
\end{proof}

\subsection{The pure cohomology of hyperelliptic mapping class groups}
In this subsection, we collect more Hodge theoretic results. Some are of independent interest and some we will need later.

\begin{proposition}\label{first homology H}For $(S,\u,\cP)$ a marked hyperelliptic surface of genus $\geq 2$, 
there holds $W_1H^1(\PU(S,\cP),\Q)=\{0\}$. In particular, $W_1H^1(\U(S,\cP),\Q)=W_1H^1(\U(S,\vec\cP),\Q)=\{0\}$.
\end{proposition}

\begin{proof}By the first item of Corollary~\ref{simplyconnected}, we have that $H^1(\mathrm{P}\ccM(S,\u,\cP),\Q)=\{0\}$ and, since 
$\mathrm{P}\ccM(S,\u,\cP)$ is a smooth compactification of $\mathrm{P}\cM(S,\u,\cP)$, there is an isomorphism: 
\[H^1(\mathrm{P}\ccM(S,\u,\cP),\Q)\cong W_1 H^1(\mathrm{P}\cM(S,\u,\cP),\Q).\]
\end{proof}

\begin{proposition}\label{W_1 no trivial}For $g(S)\geq 1$, the hyperelliptic mapping class group $\U(S,\cP)$ contains a normal finite index 
subgroup $\U^\l$ such that $W_1H^1(\U^\l)\neq \{0\}$.
\end{proposition}

\begin{proof}It is enough to prove the proposition for $g(S)=1$, $n(S)=1$ and  $g(S)\geq 2$ and $n(S)=0$, since, if
$p\co\U(S,\cP)\to\U(S)$ is the natural epimorphism, then, for a finite index subgroup 
$\U^\l$ of $\U(S,\cP)$, the natural homomorphism $p^\ast\co W_1H^1(p(\U^\l))\to W_1H^1(\U^\l)$ is injective.

For $\G(S_{1,1})=\SL_2(\Z)$, the claim is well-known since $W_1H^1(\U^\l)$
is just the first cohomology group of the projective modular curve associated to $\U^\l$. 

Let us assume then $n(S)=0$ and $g(S)\geq 2$. Let $\cT(S,\u)$ be the universal cover of the moduli stack of hyperelliptic curves $\cM(S,\u)$. 
For $g(S)=2$, this is just the ordinary genus $2$ Teichm\"uller space $\cT(S)$ associated to $S$, while, for $g(S)\geq 3$, it can be identified 
with a closed contractible submanifold of the Teichm\"uller space $\cT(S)$. Since there is a natural \'etale covering $\cM(S,\u)\to\cM(S_{0,2g+2})$, 
we have a natural isomorphism $\cT(S,\u)\cong\cT(S_{0,2g+2})$. There is also a natural epimorphism $p\co\PG(S_{0,2g+2})\to\PG(S_{0,4})$ 
and a corresponding $p$-equivariant surjective map $\phi\co\cT(S_{0,2g+2})\to\cT(S_{0,4})$.

Let us assume that $\U^\l$ is contained in the abelian level $\U(S)^{[2]}$. By Proposition~3.3 in \cite{Hyp}, there is a natural epimorphism 
$r\co\U(S)^{[2]}\to\PG(S_{0,2g+2})$. It follows that the map $\phi$ induces 
a surjective map $\cT(S,\u)/\U^\l\to\cT(S_{0,4})/p(r(\U^\l))$, which induces an epimorphism on fundamental groups. 
Therefore, we have a natural injective map:
\[W_1H^1(\cT(S_{0,4})/p(r(\U^\l)))\hookra W_1H^1(\cT(S,\u)/\U^\l)=W_1H^1(\U^\l).\]

Since $\PM(S_{0,4})\cong\P^1\ssm\{0,1,\infty\}$, there is a normal finite index subgroup $N$ of $\PG(S_{0,4})$ such that
the associated covering of $\PM(S_{0,4})$ has genus $\geq 1$. 

For $\U^\l$ any normal finite index subgroup of $\U(S,\cP)$
which is contained in $r^{-1}(p^{-1}(N))$, we then have $W_1H^1(\cT(S_{0,4})/p(r(\U^\l)))\neq \{0\}$ and the claim of the proposition follows.
\end{proof}

A consequence of independent interest of Proposition~\ref{W_1 no trivial} is a slight improvement of Corollary~1.1 in \cite{Stylianakis}  
(cf.\ \cite{Funar} and \cite{Masbaum} for the mapping class group version of this result):

\begin{corollary}\label{infinite}For every even integer $m$ and for $k\gg 0$, the normal subgroup of the hyperelliptic mapping class group 
$\U(S,\cP)$ generated by the ${m^k}$-powers of symmetric multitwists is of infinite index.
\end{corollary} 

\begin{proof}There is a normal subgroup $N$ of $\PG(S_{0,4})$ of index a power $(m/2)^k$ such that the associated covering of $\PM(S_{0,4})$ 
has genus $\geq 1$. Then, $\U^\l:=r^{-1}(p^{-1}(N))$ is a normal finite index subgroup of $\U(S)^{[2]}$ of index $(m/2)^k$. Since $\U(S)^{[2]}$ 
contains all squares of symmetric multitwists, it follows that all ${m^k}$-powers of symmetric multitwists belong to $\U^\l$. 
Let us denote by $K$ the normal subgroup of $\U^\l$ generated by these elements and
let $\ccM(S,\u,\cP)^\l$ be the DM compactification of the level structure $\cM(S,\u,\cP)^\l$ over $\cM(S,\u,\cP)$ associated to $\U^\l$.
There is an epimorphism $\U^\l/K\twoheadrightarrow H_1(\ccM(S,\u,\cP)^\l,\Z)$ and, as we showed in the proof of Proposition~\ref{W_1 no trivial}, 
there holds $H^1(\ccM(S,\u,\cP)^\l,\Z)\otimes\Q\cong W_1H^1(\U^\l)\neq\{0\}$.
\end{proof}

For $\s\in C(S,\u,\cP)$, where $S$ is a closed surface, the stabilizer $\PU(S,\cP)_{\vec\s}$ of the oriented simplex $\vec\s$ for the action of $\PU(S,\cP)$ 
is described by the short exact sequence~(\ref{stabPHyp}):
\[1\to\prod_{\g\in\s_{ns}}\tau_\g^{2\Z}\times\prod_{\g\in\s_{s}}\tau_\g^{\Z}\to\PU(S,\cP)_{\vec\s}\sr{q_\s}{\to}\PU(S\ssm\s,\cP)\to 1.\]
Since powers of Dehn twists have weight $-2$ in homology, we then have:

\begin{lemma}\label{lemma:qpurityH} For a finite index subgroup $\U^\l_{\vec\s}$ of $\PU(S,\cP)_{\vec\s}$, the restriction map:
\begin{equation}\label{restriction}
W_1 H^1(q_\s(\U^\l_{\vec\s}),\Q)\sr{q_\s^\ast}{\to} W_1 H^1(\U^\l_{\vec\s},\Q)
\end{equation}
is an isomorphism.
\end{lemma}

\subsection{A counterexample to the hyperelliptic Putman-Wieland problem}\label{counterexamples}
Let us fix some notation. For $\g$ a nonseparating simple closed curve on $S$, let $S_\g:=S\ssm\g$, let $\ol{S}_\g$ be the closed surface 
obtained from $S_\g$ filling in the punctures and put $\{Q,R\}:=\ol{S}_\g\ssm S_\g$. 

We then denote by $\ol{\U}_{\vec\g}^\l$ the image of $\U^\l\cap\PU(S,\cP)_{\vec\g}$ in $\PU(\ol{S}_\g,\cP\cup\{Q\})$ by the composition of the 
natural epimorphism $q_\g\co\PU(S,\cP)_{\vec\g}\to\PU(S_\g,\cP)$ (cf.\ the short exact sequence~(\ref{stabPHyp})) with the natural epimorphism 
$\PU(S_\g,\cP)\to\PU(\ol{S}_\g,\cP\cup\{Q\})$ (cf.\ the last one of the short exact sequences~(\ref{forgetsympunctures})). 
There is a Birman short exact sequence:
\[1\to\pi_1(\ol{S}_\g\ssm\cP,Q)\to\PU(\ol{S}_\g,\cP\cup\{Q\})\sr{f_Q}{\to}\PU(\ol{S}_\g,\cP)\to 1.\]

Let $\ol{S}_\g^\l\to\ol{S}_\g$ be the covering of closed surfaces obtained filling in the punctures of the unramified covering 
$(\ol{S}_\g\ssm\cP)^\l\to\ol{S}_\g\ssm\cP$ associated to the subgroup $\pi_1(\ol{S}_\g\ssm\cP,Q)\cap\ol{\U}_{\vec\g}^\l$ of  
$\pi_1(\ol{S}_\g\ssm\cP,Q)$. Let then:
\[L\ol{\rho}_\g^\l\co  f_Q(\ol{\U}^\l_{\vec\g})\to\Sp(H^1(\ol{S}_\g^\l,\Q))\]
be the corresponding linear representation.

\begin{definition}\label{1-bounding}Let $\mathcal{D}_1(S,\u,\cP)$ be the normal subgroup of $\U(S,\cP)$ generated by all Dehn twists about
separating curves bounding a genus $1$ subsurface of $S$ carrying no marked points.
\end{definition}

We then have:

\begin{theorem}\label{No Hyp Putman}For $(S,\u,\cP)$ a marked hyperelliptic closed surface of genus $g\geq 3$, let $\U^\l$ be a finite index 
subgroup of $\PU(S,\cP)$ which is normalized by $\U(S,\vec\cP)$, contains $\mathcal{D}_1(S,\u,\cP)$ and is such that $W_1 H^1(\U^\l,\Q)\neq 0$.
Then, for any nonseparating simple closed curve $\g$ on $S$, the associated linear representation 
$L\ol{\rho}_\g^\l\co  f_Q(\ol{\U}^\l_{\vec\g})\to\Sp(H^1(\ol{S}_\g^\l,\Q))$ has a finite nontrivial orbit.
\end{theorem}

\begin{proof}The proof is by contraposition. So let us assume that, for a given finite index subgroup $\U^\l$ of $\PU(S,\cP)$, which satisfies 
the hypotheses of the theorem, and some nonseparating simple closed curve $\g$ on $S$, the linear representation 
$L\ol{\rho}_\g^\l\co  f_Q(\ol{\U}^\l_{\vec\g})\to\Sp(H^1(\ol{S}_\g^\l,\Q))$ has only infinite nontrivial orbits. 
Note that, since $\U^\l$ is normal in $\U(S,\vec\cP)$ and nonseparating  simple closed curve on $S$ form a single orbit under the latter group, 
this property then holds for all of them. We will show that this implies that $W_1 H^1(\U^\l)= 0$.

By Proposition~\ref{first homology H}, we have that $W_1H^1(\U^\l,\Q)^{\U(S,\vec\cP)}=W_1H^1(\U(S,\vec\cP),\Q)=0$. 
So it is enough to prove that our hypothesis implies that $\U(S,\vec\cP)$ acts trivially on $W_1 H^1(\U^\l)$. 
Since, by Proposition~\ref{generators}, the hyperelliptic mapping class group $\U(S,\vec\cP)$ is generated by 
symmetric nonseparating Dehn twists, we just need to show that, for any symmetric nonseparating simple closed curve $\g$ on $S$, the Dehn twist 
$\tau_\g$ acts trivially on $W_1 H^1(\U^\l)$. 

Note that $S_\g:=S\ssm\g$ is homeomorphic to $S_{g-1,2}$ with the two symmetric punctures, labeled
by $Q$ and $R$, which correspond to the two orientations of $\g$. The image $q_\g(\U^\l_{\vec\g})$ of $\U^\l_{\vec\g}$ 
under the map $q_\g\co\PU(S,\cP)_{\vec\g}\to\PU(S_\g,\cP)$  is a finite index subgroup and, by~(\ref{restriction}), there is an isomorphism:
\begin{equation}\label{iso1}
q_\g^\ast\co W_1H^1(\U^\l_{\vec\g},\Q)\sr{\sim}{\to} W_1H^1(q_\g(\U^\l_{\vec\g}),\Q).
\end{equation} 

The surface $\ol{S}_\g$ is obtained from $S_\g$ by filling in the two punctures bounded by $\g$, so that $\ol{S}_\g\cong S_{g-1}$. 
By the bottom short exact sequence~(\ref{forgetsympunctures}), there is an epimorphism $\phi_\g\co\PU(S_\g,\cP)\to\PU(\ol{S}_\g,\cP\cup\{Q\})$,
with kernel $N_\g$ the normal subgroup generated by the Dehn twists about symmetric separating simple closed curves on $S_\g$ bounding 
a disc which only contains the punctures labeled by $Q$ and $R$ and no marked points. 

Since $N_\g$ is contained in 
$q_\g(\PU(S,\cP)_{\vec\g}\cap\mathcal{D}_1(S,\u,\cP))$ and we assumed that $\mathcal{D}_1(S,\u,\cP)\subset\U^\l$, it follows that
$N_\g$ is contained in $q_\g(\U^\l_{\vec\g})$. There is then a short exact sequence:
\[1\to N_\g\to q_\g(\U^\l_{\vec\g})\sr{\phi_\g}{\to}\ol{\U}^\l_{\vec\g}\to 1.\]
Since the elements of $N_\g$ map to elements of weight $-2$ in $H_1(q_\g(\U^\l_{\vec\g}),\Q)$, it follows that $\phi_\g$ induces an isomorphism:
\begin{equation}\label{iso2}
\phi_\g^\ast\co W_1H^1(\ol{\U}^\l_{\vec\g},\Q)\sr{\sim}{\to} W_1H^1(q_\g(\U^\l_{\vec\g}),\Q).
\end{equation} 

Let us now consider the short exact sequence:
\[1\to\pi_1(\ol{S}_\g\ssm\cP,Q)\to\PU(\ol{S}_\g,\cP\cup\{Q\})\sr{f_Q}{\to}\PU(\ol{S}_\g,\cP)\to 1.\]
From the hypothesis that the linear representation $L\ol{\rho}_\g^\l\co  f_Q(\ol{\U}^\l_{\vec\g})\to\Sp(H^1(\ol{S}_\g^\l,\Q))$ 
has only infinite nontrivial orbits and the short exact sequence~(\ref{purexact}), it follows that the epimorphism $f_Q$ induces an isomorphism: 
\begin{equation}\label{iso3}
f_Q^\ast\co W_1H^1(f_Q(\ol{\U}^\l_{\vec\g}),\Q)\sr{\sim}{\to} W_1H^1(\ol{\U}^\l_{\vec\g},\Q).
\end{equation} 

Since $(\U^\l_\g:\U^\l_{\vec\g})\leq 2$, the inclusion $\U^\l_{\vec\g}\subseteq\U^\l_\g$ induces a monomorphism on cohomology
$W_1H^1(\U^\l_{\g},\Q)\hookra W_1H^1(\U^\l_{\vec\g},\Q)$. Composing with the isomorphism~(\ref{iso1}) 
and the inverses of the isomorphisms~(\ref{iso2}) and~(\ref{iso3}), we get a monomorphism:
\begin{equation}\label{iso4}
W_1H^1(\U^\l_{\g},\Q)\hookra W_1H^1(f_Q(\ol{\U}^\l_{\vec\g}),\Q).
\end{equation} 

Let $\delta$ be a simple closed curve on $S$ with the property that it intersects $\g$ transversally in a single point and let
$\alpha$ be the boundary of a tubular neighborhood of $\delta\cup\g$ in $S$. Then, $\alpha$ is a separating simple closed curve 
on $S$ and we let $S_\alpha$ be the component of $S\ssm\alpha$ of genus $>1$. The natural embedding $S_\alpha\subset\ol{S}_\g$
is homotopic to the embedding $S_\alpha\subset\ol{S}_\alpha$, where $\ol{S}_\alpha$ is the closed surface obtained from $S_\alpha$ filling in the hole 
bounded by $\alpha$ with a point. Therefore, we can identify $\PU(\ol{S}_\alpha,\cP)$ with $\PU(\ol{S}_\g,\cP)$. Moreover, by the definition of
pure hyperelliptic mapping class group of a marked hyperbolic hyperelliptic surface, the group $\PU(S_\alpha,\cP)$ identifies with the 
group $\PU(\ol{S}_\alpha,\cP)$. 

By the short exact sequence~(\ref{stabPHyp}), there is a natural epimorphism onto $\PU(S_\alpha,\cP)$ from the stabilizer 
$\PU(S,\cP)_\alpha=\PU(S,\cP)_{\vec\alpha}$ of the isotopy class of $\alpha$. Let then: 
\[r_\alpha\co\PU(S,\cP)_\alpha\to\PU(\ol{S}_\g,\cP)\]
be the epimorphism obtained from this epimorphism and the identification of $\PU(S_\alpha,\cP)$ with $\PU(\ol{S}_\alpha,\cP)$ and of 
$\PU(\ol{S}_\alpha,\cP)$ with $\PU(\ol{S}_\g,\cP)$.
If $\U^\l_\alpha$ is the stabilizer in $\U^\l$ of the isotopy class of $\alpha$, its image $r_\alpha(\U^\l_\alpha)$ is then a finite index
subgroup of $f_Q(\ol{\U}^\l_{\vec\g})$ and
restriction induces a monomorphism $W_1H^1(f_Q(\ol{\U}^\l_{\vec\g}),\Q)\hookra W_1H^1(r_\alpha(\U^\l_\alpha),\Q)$, which,
composed with the monomorphism~(\ref{iso4}), gives the natural monomorphism:
\[W_1H^1(\U^\l_{\g},\Q)\hookra W_1H^1(r_\alpha(\U^\l_\alpha),\Q).\]
Since we may interchange $\g$ and  $\delta$, we also obtain a natural monomorphism:
\[W_1H^1(\U^\l_{\delta},\Q)\hookra W_1H^1(r_\alpha(\U^\l_\alpha),\Q).\]

The compositions
\[\begin{array}{ll}
&W_1H^1(\U^\l,\Q)\to W_1H^1(\U^\l_\g,\Q)\hookra W_1H^1(r_\alpha(\U^\l_\alpha),\Q)\\
\mbox{and}\hspace{2cm}&\\
&W_1H^1(\U^\l,\Q)\to W_1H^1(\U^\l_\delta,\Q)\hookra W_1H^1(r_\alpha(\U^\l_\alpha),\Q)
\end{array}\]
then define the same map. Therefore, $W_1H^1(\U^\l,\Q)\to W_1H^1(\U^\l_\g,\Q)$ and
$W_1H^1(\U^\l,\Q)\to W_1H^1(\U^\l_\delta,\Q)$ have the same kernel. Since, by Proposition~\ref{Farey}, 
any two symmetric nonseparating closed curves on $S$ can be connected by a sequence of such curves whose successive terms meet transversally 
in a single point, we see that this kernel is independent of $\g$. 

By Corollary~\ref{hyp complex}, the geometric realization $|C_\mathrm{ns}(S,\u,\cP)|$ of the complex of 
symmetric nonseparating curves is connected. Hence, the low terms exact sequence associated to the first spectral sequence of equivariant 
cohomology yields the exact sequence
\[0\to H^1(|C_\mathrm{ns}(S,\u,\cP)|/\U^\l,\Q)\to H^1(\U^\l,\Q)\to \bigoplus_{\g}H^1(\U^\l_\g,\Q),\]
where $\g$ runs over a set of representatives of the $\U^\l$-orbits of isotopy classes of nonseparating simple
closed curves. 

Since the natural monomorphism $H^1(|C_\mathrm{ns}(S,\u,\cP)|/\U^\l,\Q)\hookra H^1(\U^\l,\Q)$
has image contained in the weight zero subspace of $H^1(\U^\l,\Q)$, we then get the exact sequence:
\[0\to W_1H^1(\U^\l,\Q)\to\bigoplus_\g W_1H^1(\U^\l_\g,\Q).\]
Above we proved that, for all symmetric nonseparating simple closed curves $\g$, the maps 
$W_1 H^1(\U^\l,\Q)\to W_1 H^1(\U^\l_\g,\Q)$ have the same kernel. It follows that $W_1 H^1(\U^\l,\Q)\to W_1 H^1(\U^\l_\g,\Q)$ 
is injective. Since the Dehn twist $\tau_\g$ centralizes the subgroup $\U^\l_\g$ of $\U(S,\cP)$ and then acts trivially on $W_1 H^1(\U^\l_\g,\Q)$, 
it does so on $W_1 H^1(\U^\l,\Q)$ too, which completes the proof of the theorem.
\end{proof}

The following result shows that, at least for $\sharp\,\cP\geq 4$, Theorem~\ref{No Hyp Putman} is not empty:

\begin{theorem}\label{Noempty}For $(S,\u,\cP)$ a marked hyperelliptic closed surface of genus $g\geq 2$ with $\sharp\,\cP\geq 4$,
there is a finite index subgroup $\U^\l$ of $\PU(S,\cP)$ which is normalized by $\U(S,\vec\cP)$, contains $\mathcal{D}_1(S,\u,\cP)$ 
and is such that $W_1 H^1(\U^\l,\Q)\neq 0$.
\end{theorem}

\begin{proof}For $\cP\subset\cP'$, there is an epimorphism $f_{\cP'\ssm\cP}\co\PU(S,\cP')\to\PU(S,\cP)$, induced by
forgetting the points $\cP'\ssm\cP$, and so a monomorphism
$W_1 H^1(\U^\l,\Q)\hookra W_1 H^1(f_{\cP'\ssm\cP}^{-1}(\U^\l),\Q)$. Therefore,  it is not restrictive to assume that $\sharp\,\cP=4$.

Let then $\{Q_i,R_i\}$, for $i=1,\ldots,4$, be four distinct symmetric pairs of points on $S$, put $\cP:=\{Q_1,\ldots,Q_4\}$ and let
$S':=S\ssm\cup_{i=1}^4\{Q_i,R_i\}$. By Theorem~\ref{Q}, there is a natural short exact sequence:
\[1\to N_\circ\to\PU(S')\sr{f_\circ}{\to}\PU(S,\cP)\to 1,\]
where $N_\circ$ is the normal subgroup generated by the Dehn twists about symmetric separating simple closed curves on $S'$
bounding a disc which only contains the punctures labeled by a pair of points $\{Q_i,R_i\}$, for $i=1,\ldots,4$.

By the short exact sequence~(\ref{PHypEx}), there is also a natural isomorphism:
\[\PU(S')\cong\PG(S'_{/\u},\cB_\u')\cong\PG(S_{0,2g+6}),\]
where $\cB_\u'$ is the branch locus of the quotient map $S'\to S_{/\u}'$ (consisting of $2g+2$ points). 

There is then a natural epimorphism (induced by forgetting the branch locus $\cB_\u'$):
\[f_{\cB_\u'}\co\PU(S')\to\PG(S'_{/\u})\cong\PG(S_{0,4}).\]

Note that the kernel of $f_{\cB_\u'}$ contains $N_\circ$. In fact, the isomorphism $\PU(S')\cong\PG(S'_{/\u},\cB_\u')$ maps a Dehn twist about 
a symmetric separating simple closed curve on $S'$, bounding a disc which only contains the punctures labeled by the pair $\{Q_i,R_i\}$ from $S$, 
to a Dehn twist about a simple closed curve on $S'_{/\u}$ bounding a disc which only contains the puncture of $S'_{/\u}$ in the image of $\{Q_i,R_i\}$ 
and some point of the branch locus $\cB_\u'$. 

Therefore, the epimorphism $f_{\cB_\u'}$ factors through a natural epimorphism:
\[\bar f_{\cB_\u'}\co\PU(S,\cP)\to\PG(S'_{/\u})\cong\PG(S_{0,4}).\]

Note also that the above isomorphism $\PU(S')\cong\PG(S'_{/\u},\cB_\u')$ maps a Dehn twist about a separating simple closed curve bounding a genus $1$ 
unpunctured subsurface of $S'$ to a Dehn twist about a simple closed curve on $S'_{/\u}$ bounding an unpunctured disc containing three points of the
branch locus $\cB_\u'$. Therefore, the subgroup $\mathcal{D}_1(S',\u)$ is in the kernel of $f_{\cB_\u'}$ which implies that $\mathcal{D}_1(S,\u,\cP)$ 
is in the kernel of $\bar f_{\cB_\u'}$.

Let $\G^\l$ be a subgroup of $\PG(S'_{/\u})\cong\PG(S_{0,4})$ such that the associated level structure $\mathrm{P}\cM(S'_{/\u})^\l$ is a curve
of genus $\geq 1$ (so that $W_1 H^1(\G^\l,\Q)\neq 0$). Since $\U(S,\vec\cP)\cap\bar f_{\cB_\u'}^{-1}(\G^\l)$ is a subgroup of finite index in $\U(S,\vec\cP)$, 
it contains a subgroup $\U^\l$ which is maximal for the property of being normal and of finite index in $\U(S,\vec\cP)$.
Since $\mathcal{D}_1(S,\u,\cP)$ is a normal subgroup of $\U(S,\vec\cP)$ 
contained in $\bar f_{\cB_\u'}^{-1}(\G^\l)$, the subgroup $\U^\l$ contains $\mathcal{D}_1(S,\u,\cP)$.

As we observed in Section~\ref{HypTopType}, the isomorphism $\PU(S')\cong\PG(S'_{/\u},\cB_\u')$ is induced by an isomorphism~(\ref{gerbes}) 
$\cM(S',\u)\cong\cM(S'_{/\u},\cB_\u')$ of moduli stacks. Similarly, the epimorphism $f_{\cB_\u'}$ is induced by the smooth morphism with
connected fibers $\cM(S'_{/\u},\cB_\u')\to\cM(S'_{/\u})$. 

By the proof of Theorem~\ref{Q}, the epimorphism $f_\circ\co f_\circ^{-1}(\U^\l)\to\U^\l$ is induced by an open embedding of 
level structures $\cM(S',\u)^{\l'}\hookra\cM(S,\u,\cP)^{\l}$, where we put $\U^{\l'}:=f_\circ^{-1}(\U^\l)$. So, the epimorphism $f_\circ$ induces
an isomorphism $f_\circ^\ast\co W_1 H^1(\U^\l,\Q)\sr{\sim}{\to}W_1 H^1(f_\circ^{-1}(\U^\l),\Q)$. 
We conclude that there is a monomorphism 
$(f_\circ^\ast)^{-1}\circ f_{\cB_\u'}^\ast\co W_1 H^1(\G^\l,\Q)\hookra W_1 H^1(\U^\l,\Q)$ and, hence, $W_1 H^1(\U^\l,\Q)\neq 0$.
\end{proof}


\begin{corollary}\label{NoHypPutWie}With the notations of Section~\ref{HypProblem}, for $(S,\u,\cP)$ a marked hyperelliptic closed surface of genus 
$g\geq 2$ with $\sharp\,\cP\geq 4$ and a point $Q\in S\ssm\cP$, there is a finite index subgroup $\U^\l$ of the hyperelliptic mapping class group
$\U(S,\cP,Q)$ such that the associated linear representation $L\rho^\l\co  f_Q(\U^\l)\to\Sp(H^1(S^\l,\Q))$ has a finite nontrivial orbit.
\end{corollary}

\begin{remark}\label{explicit}
Based on local monodromy computations 
(cf.\ Theorem~2.2 in \cite{B2}), possible candidates for $\U^\l$ are the finite index subgroups defined by the intersection of $\PU(S,\cP)$
with the finite index subgroups $\G^{w(4,3)}$ and $\G^{w(4,6)}$ of $\PG(S,\cP)$ defined in Section~2 of \cite{B2}.
\end{remark}

\subsection{A counterexample to the Putman-Wieland conjecture for $g(S)=2$}
In particular, we get  a counterexample to the $n$-punctured ($n\geq 4$), genus $2$ case of the Putman-Wieland 
conjecture\footnote{After completing a draft of this paper, I learned that Markovi\'c (cf.\ \cite{Markovic}), 
with a different approach, found a counterexample to the Putman-Wieland conjecture in genus $2$ for $\sharp\,\cP\geq 0$.}:

\begin{corollary}\label{NoPutWie2}For $(S,\cP)$ a marked closed surface of genus $2$ with $\sharp\,\cP\geq 4$ and $Q\in S\ssm\cP$, 
there is a finite index subgroup $\G^\l$ of the mapping class group $\G(S\ssm\cP,Q)$ such that the associated linear representation
$L\rho^\l\co  f_Q(\G^\l)\to\Sp(H^1(S^\l,\Q))$ has a finite nontrivial orbit.
\end{corollary}

\section{Profinite hyperelliptic mapping class groups}\label{profinite}
\subsection{Profinite and congruence topologies on mapping class groups}
The \emph{profinite topology} on a group $G$ is defined taking for basis of open subsets the cosets of all normal subgroups of finite index.
In particular, every finite index subgroup $H$ of $G$ is an open subset for the profinite topology and we write $H<_o G$.
The \emph{profinite completion $\wh{G}$ of $G$} is the completion of $G$ with respect to the profinite topology. 
There is a natural isomorphism of topological groups $\wh{G}\cong\varprojlim_{N\lhd_o G}G/N$, where each finite quotient $G/N$ is endowed
with the discrete topology.

More generally, every cofiltered system $\{N^\l\}_{\l\in\L}$ of normal finite index subgroups of $G$ defines a topology on $G$ as above and
the \emph{pro-$\L$ completion} $\wh{G}^\L$ is naturally isomorphic to the inverse limit $\varprojlim_{\l\in\L}G/N^\l$.

For instance, on the mapping class group $\G(S,\cP)$, besides the profinite topology, a natural topology is defined in the following way. 

For a fixed $Q\in S\ssm\cP$, there is a natural faithful representation:
\[\G(S,\cP)\hookra\Out(\pi_1(S\ssm\cP,Q)).\]
We then associate to every $\G(S,\cP,Q)$-invariant finite index subgroup $K$ of $\pi_1(S\ssm\cP,Q)$ the kernel $\G^K$ 
of the induced representation $\G(S,\cP)\to\Out(\pi_1(S\ssm\cP,Q)/K)$. This is a normal finite index subgroup of $\G(S,\cP)$ called the 
\emph{geometric level associated to $K$}. 

Since, in a finitely generated group, characteristic finite index subgroups are cofinal in the system
of all finite index subgroups, the set of all geometric levels $\{\G^K\}$, ordered by inclusion, is a cofiltered system. The associated topology on 
the mapping class group $\G(S,\cP)$ is then called the \emph{congruence topology} and the associated completion $\kG(S,\cP)$ the 
\emph{congruence completion} or also the \emph{procongruence mapping class group}.

We define the congruence topology on a subgroup $\G$ of $\G(S,\cP)$ to be the topology induced by the congruence topology of $\G(S,\cP)$. 
The completion $\kG$ with respect to this topology is the \emph{congruence completion} of $\G$ and identifies with a subgroup of $\kG(S,\cP)$.
We call the congruence completions $\kU(S,\cP)$, $\kU(S,\vec\cP)$ and $\kPU(S,\cP)$ of $\U(S,\cP)$, $\U(S,\vec\cP)$ and $\PU(S,\cP)$ also
\emph{procongruence (pure) hyperelliptic mapping class groups}.

\subsection{The congruence subgroup problem}
By a classical result of Grossman (cf.\ \cite{Grossman}), we know that the natural homomorphism $\G(S,\cP)\to\Out(\hp_1(S\ssm\cP,Q))$
is injective, so that the natural homomorphism with dense image $\G(S,\cP)\to\kG(S,\cP)$ is also injective. A much more difficult question
is whether the natural homomorphism:
\[\hG(S,\cP)\to\Out(\hp_1(S\ssm\cP,Q))\]
is injective. This is known as the \emph{congruence subgroup problem for mapping class groups}. A positive answer is only known for $g(S)\leq 2$
(cf.\ \cite{Asada}, for $g(S)=1$, and \cite{Hyp}, \cite{HI}, \cite{congtop}, for $g(S)=2$). These are particular cases of the more general result
(cf.\ Theorem~1.4 in \cite{congtop}):

\begin{theorem}\label{conghypclosed}Let $(S,\u,\cP)$ be a hyperbolic marked hyperelliptic \emph{closed} surface. Then, the congruence subgroup 
property holds for the associated hyperelliptic mapping class groups, that is to say, there is a natural isomorphism $\hU(S,\cP)\cong\kU(S,\cP)$.
\end{theorem}

\subsection{The congruence subgroup property for hyperelliptic mapping class groups}
Theorem~\ref{conghypclosed} can be extended to an arbitrary hyperbolic marked hyperelliptic surface:

\begin{theorem}\label{conghyp}Let $(S,\u,\cP)$ be a hyperbolic marked hyperelliptic surface. Then, the congruence subgroup 
property holds for the associated hyperelliptic mapping class group, that is to say, there is a natural isomorphism $\hU(S,\cP)\cong\kU(S,\cP)$.
\end{theorem}

\begin{proof}Since $\U(S,\cP)$ is a finite index subgroup of $\U(\ring{S},\cP)$, it is not restrictive to assume that $\partial S=\emptyset$.
Let us consider first the case when $S$ is not hyperbolic, that is to say, either $S$ is closed of genus $1$ and $\sharp\,\cP\geq 1$ or 
$g(S)=0$ and $n(S)\leq 2$. The former case is already dealt by Theorem~\ref{conghypclosed} and the latter easily follows from it.

Let us then assume that $S$ is hyperbolic.
Since $\PU(S,\cP)$ is a finite index subgroup of $\U(S,\cP)$, it is enough to prove the congruence subgroup property for $\PU(S,\cP)$.
By the Birman exact sequence for procongruence mapping class groups (cf.\ Corollary~4.7 in \cite{congtop}), a standard argument reduces 
the proof to the case when $S$ is hyperbolic and $\cP=\emptyset$. 

Note that there is then a natural isomorphism $\PU(S)\cong\PU(S\ssm\cW_\cP(S))$ and the epimorphism $\hPU(S)\to\kPU(S)$ factors
through the induced isomorphism $\hPU(S)\cong\hPU(S\ssm\cW_\cP(S))$ and the epimorphisms $\hPU(S\ssm\cW_\cP(S))\to\kPU(S\ssm\cW_\cP(S))$
and $\kPU(S\ssm\cW_\cP(S))\to\kPU(S)$.

Therefore, the congruence subgroup property for $\PU(S)$ implies that the congruence subgroup property holds
for $\PU(S\ssm\cW_\cP(S))$ as well. We can then also assume that 
$\cW_\circ(S)=\emptyset$ and, in conclusion, that $(S,\u)$ is a hyperbolic surface without boundary, no Weierstrass punctures 
and at least two symmetric punctures (since the closed surface case is already dealt by Theorem~\ref{conghypclosed}). 

\begin{lemma}\label{reduction}For $Q\in S\ssm\cW_\cP(S)$, the natural representation 
\[\hPU(S\ssm\cW_\cP(S))\to\Out(\hp_1(S\ssm\cW_\cP(S),Q))\]
is faithful.
\end{lemma}

\begin{proof}Since, by our hypothesis, $S$ has at least two symmetric punctures, there is a natural isomorphism 
$\hPU(S\ssm\cW_\cP(S))\cong\hPG(S_{/\u}\ssm\cB_\u)$ (cf.\ the short exact sequence~(\ref{PHypEx}) and following remarks). 
So it is enough to show that the representation: 
\begin{equation}\label{genus0red}
\hPG(S_{/\u}\ssm\cB_\u)\to\Out(\hp_1(S\ssm\cW_\cP(S),Q)),
\end{equation}
obtained composing this isomorphism with the given representation, is faithful. 
Let $u\in\pi_1(S\ssm\cW_\cP(S),Q)$ be represented by a loop whose free isotopy class contains a simple closed curve about
one of the punctures obtained removing the Weierstrass points and let $\Aut_u(\hp_1(S\ssm\cW_\cP(S),Q))$ be the subgroup 
of $\Aut(\hp_1(S\ssm\cW_\cP(S),Q))$ consisting of those automorphisms which fix the element $u$. 
There is then a natural homomorphism:
\[\hat{\Theta}_u\co\Aut_u(\hp_1(S\ssm\cW_\cP(S),Q))\left/\Inn(u^\ZZ)\right.\to\Out(\hp_1(S\ssm\cW_\cP(S),Q))\]
which, by Lemma~2.2 in \cite{congtop}, is injective. Since the representation~(\ref{genus0red}) factors through 
the homomorphism $\hat{\Theta}_u$ and a natural representation:
\begin{equation}\label{genus0red2}
\hPG(S_{/\u}\ssm\cB_\u)\to\Aut_u(\hp_1(S\ssm\cW_\cP(S),Q))\left/\Inn(u^\ZZ)\right.,
\end{equation}
it is then enough to prove that the representation~(\ref{genus0red2}) is faithful. 

Let us observe that the fundamental group $\pi_1(S\ssm\cW_\cP(S),Q)$ identifies with an index $2$ subgroup of $\pi_1(S_{/\u}\ssm\cB_\u,\bar Q)$,
where $\bar Q$ is the image of the base point $Q$ by the covering map $S\ssm\cW_\cP(S)\to S_{/\u}\ssm\cB_\u$. 
Then, by Lemma~8 in \cite{Asada}, the natural representation:
\[R_\u\co\Aut_u^\sharp(\hp_1(S_{/\u}\ssm\cB_\u,\bar Q))\left/\Inn(u^\ZZ)\right.\to\Aut_u(\hp_1(S\ssm\cW_\cP(S),Q))\left/\Inn(u^\ZZ)\right.\]
is faithful, where we denote by $\Aut_u^\sharp(\hp_1(S_{/\u}\ssm\cB_\u,\bar Q))$ the subgroup of $\Aut_u(\hp_1(S_{/\u}\ssm\cB_\u,\bar Q))$
consisting of those elements which preserve the subgroup $\hp_1(S\ssm\cW_\cP(S),Q)$. 

The representation~(\ref{genus0red2}) then factors through the representation $R_\u$ and a natural representation:
\[\hPG(S_{/\u}\ssm\cB_\u)\to\Aut_u^\sharp(\hp_1(S_{/\u}\ssm\cB_\u,\bar Q))\left/\Inn(u^\ZZ)\right..\]

Thus, it is enough to prove that this last representation is faithful, which follows by the ordinary case of the congruence subgroup property for genus $0$
pure mapping class groups (cf.\ Theorem~1 in \cite{Asada}).
\end{proof}

From the natural isomorphism $\hPU(S)\cong\hPU(S\ssm\cW_\cP(S))$ and Lemma~\ref{reduction}, it follows that, 
for $Q\in S\ssm\cW_\cP(S)$, there is a natural faithful representation:
\[\hPU(S)\hookrightarrow\Out(\hp_1(S\ssm\cW_\cP(S),Q)).\]

In particular, the natural homomorphism $\hPU(S)\to\kPG(S\ssm\cW_\cP(S))$ is injective. In order to prove the theorem,
we then have to show that its composition with the natural epimorphism $f_{\cW_\cP}\co\kPG(S\ssm\cW_\cP(S))\to\kPG(S)$ is still injective.

Let us denote by $S(n)$ the configuration space of $n$ distinct points on $S$, by $\Pi(n)$ its fundamental group and by $\hP(n)$ the profinite
completion of this group. From Theorem~1 in \cite{Asada} and the Birman short exact sequences for procongruence mapping class groups
(cf.\ Corollary~4.7 in \cite{congtop}), it follows that the kernel of $f_{\cW_\cP}$ identifies with the group
$\hP(2g(S)+2)$. Hence, the theorem follows if we show that the image of 
$\hPU(S)\to\kPG(S\ssm\cW_\cP(S))$ has trivial intersection with the subgroup $\hP(2g(S)+2)$. 

Since the image of $\hPU(S)$
in $\kPG(S\ssm\cW_\cP(S))$ is centralized by the hyperelliptic involution $\u$, it is enough to prove that the centralizer
of $\u$ in $\kPG(S\ssm\cW_\cP(S))$ has trivial intersection with the subgroup $\hP(2g(S)+2)$. The last claim can be proved by the
same argument given at the end of the proof of Lemma~3.7 in \cite{Hyp} or, alternatively, it is an immediate consequence of Lemma~6.6 in \cite{congtop}.
\end{proof}

\subsection{Hyperelliptic mapping class groups are good}
In Proposition~3.2 in \cite{Hyp}, we showed that the hyperelliptic mapping class group $\U(S,\vec\cP)$ of a marked hyperelliptic closed 
surface is \emph{good}, that is to say, according to Serre's definition (cf.\ Exercise~1, Section~2.6 in \cite{SerreGalois}), the natural homomorphism 
$\U(S,\vec\cP)\to\hU(S,\vec\cP)$ induces an isomorphism on cohomology with finite coefficients. It is not difficult 
to extend this result to arbitrary hyperbolic marked hyperelliptic surfaces:

\begin{theorem}\label{goodhyp}Let $(S,\u,\cP)$ be a hyperbolic marked hyperelliptic surface. Then, the natural homomorphism 
$\U(S,\cP)\to\hU(S,\cP)$ induces, for every discrete, finite, continuous $\hU(S,\cP)$-module $M$, an isomorphism 
$H^i(\hU(S,\cP),M)\sr{\sim}{\to} H^i(\U(S,\cP),M)$, for all $i\geq 0$.
\end{theorem}

\begin{proof}Since the pure hyperelliptic mapping class group has finite index in the hyperelliptic mapping class group $\U(S,\cP)$,
by Hochschild-Serre spectral sequence, it is enough to show that $\PU(S,\cP)$ has this property. 

It is well known (cf.\ also the proof of Theorem~\ref{conghyp}) that the Birman short exact sequence~(\ref{BirmanExact}) 
induces one of profinite completions:
\begin{equation}\label{BirmanExactPro}
1\to\hp_1(S\ssm\cP,Q)\to\hPU(S,\cP\cup\{Q\})\to\hPU(S,\cP)\to 1.
\end{equation}

Since the fundamental group of a surface is good, by Hochschild-Serre spectral sequence, we have that, if $\PU(S,\cP)$ is good, then
$\PU(S,\cP\cup\{Q\})$ has the same property. By a simple inductive argument, we are then reduced to prove that $\PU(S)$ is good.

The short exact sequence~(\ref{PHypEx}) implies that $\PU(S)$ is good, if and only if, $\PG(S_{/\u},\cB_\u)$ is good. This is well known
(again by the profinite Birman short exact sequence and Hochschild-Serre spectral sequence), which completes the proof of the theorem.
\end{proof}

From Exercise~2, Section~2.6 in \cite{SerreGalois}, it then follows:

\begin{corollary}\label{HypBirmanExactPro}The hyperelliptic Birman exact sequences~(\ref{HypBirmanEx}) induce
the short exact sequences, on profinite completions:
\begin{equation}\label{HypBirmanExPro}\begin{array}{ll}
&1\to\hp_1(S\ssm\cW_\cP(S),Q)\to\hU(S^\circ)_\circ\to\hU(S)\to 1\\
\mbox{and}\hspace{0.5cm}&\\
&1\to\hp_1(S\ssm\cW_\cP(S),Q)\to\hPU(S^\circ)\to\hPU(S)\to 1.
\end{array}
\end{equation}
\end{corollary}

\subsection{The complex of profinite curves}
Let us recall a few definitions and results from sections~3 and~4 of \cite{B1}.
Let $G\to G'$ be a homomorphism of an abstract group in a profinite group, with dense image, 
and let $\{G^\l\}_{\l\in\L}$ be the cofiltered system formed by the sugroups of $G$ which are inverse images of open 
subgroups of $G'$. Let $X_\bt$ be a simplicial set endowed with a geometric $G$-action (see Section~3 in \cite{B1}) such that 
the set of $G$-orbits in $X_n$ is finite for all $n\geq 0$. Assume, moreover, that there exists a 
$\mu\in\L$ such that $X_\bt$, with the induced $G^\mu$-action, is a simplicial $G^\mu$-set. 

The $G'$-completion of $X_\bt$ is defined to be the simplicial profinite set
\[X_\bt':=\varprojlim_{\l\in\L,\,\l\leq\mu}X_\bt/G^\l,\]
which is then endowed with a natural continuous geometric action of $G'$. The simplicial profinite 
set $X'_\bt$ has the following universal property. Let $f\co X_\bt\to Y_\bt$ be a simplicial 
$G$-equivariant map, where $Y_\bt$ is a simplicial profinite set endowed with a continuous 
geometric action of $G'$. Then $f$ factors uniquely through the natural map
$X_\bt\to X_\bt'$ and a simplicial $G'$-equivariant continuous map $f'\co X_\bt'\to Y_\bt$.

For $S$ a hyperbolic surface, by Proposition~2.2 in \cite{B1}, there is then an order of the vertex set of 
the simplicial complex $C(S)$ such that the associated simplicial set $C(S)_\bt$ is a $\G(S)^{[4]}$-simplicial set 
endowed with a natural geometric $\G(S)$-action.  For this order, the conditions prescribed in order to define
the $\kG(S)$-completion of the simplicial set $C(S)_\bt$ are all satisfied. 
The \emph{procongruence curve complex} $\kC(S)_\bt$ is then the simplicial profinite set defined as the
$\kG(S)$-completion of $C(S)_\bt$. It comes with a natural injective $\G(S)$-equivariant map 
$C(S)_\bt\subseteq\kC(S)_\bt$ (cf.\ Proposition~5.1 in \cite{PFT}).

For $\G^\l$ a finite index subgroup of $\G(S)$ contained in the abelian level $\G(S)^{[4]}$, the quotient
$C^\l(S)_\bt:=C(S)_\bt/\G^\l$ identifies with the nerve of the DM boundary of $\ccM(S)^\l$.
Therefore the procongruence curve complex $\kC(S)_\bt$ can also be defined to be the inverse limit of the
nerves of the DM boundaries of all geometric level structures over $\cM(S)$.

A \emph{simplicial profinite complex} is an abstract simplicial complex whose 
set of vertices is endowed with a profinite topology such that the sets of $k$-simplices, 
with the induced topologies, are compact and then profinite, for all $k\geq 0$. For these simplicial 
complexes, the procedure, which associates to an abstract simplicial complex and an order of
its vertex set a simplicial set, produces a simplicial profinite set.
By Theorem~4.2 of \cite{B1}, $\kC(S)_\bt$ has a realization as the simplicial profinite set associated to
an abstract simplicial profinite complex which is defined in terms of the fundamental group $\pi_1(S,Q)$ and its profinite
completion $\hp_1(S,Q)$ as follows. 

Let $\cL(S)$ be the set of isotopy classes of simple closed curves on $S$. 
Let then $\pi_1(S,Q)/\!\sim$ be the set of conjugacy classes of elements of $\pi_1(S,Q)$ and let $\cP_2(\pi_1(S,Q)/\!\sim)$ be the set of 
unordered pairs of elements of $\pi_1(S,Q)/\!\sim$. For a given $\g\in\pi_1(S,Q)$, let us denote by $\g^{\pm 1}$ the set $\{\g,\g^{-1}\}$
and by $[\g^{\pm 1}]$ its equivalence class in $\cP_2(\pi_1(S,Q)/\!\sim)$. There is a natural embedding 
$\iota\co\cL(S)\hookra\cP_2(\pi_1(S,Q)/\!\sim)$ defined by choosing, for $\g\in\cL(S)$, an element $\vec{\g}_\ast\in\pi_1(S,Q)$ 
whose free isotopy class contains $\g$ and letting $\iota(\g):=[\vec{\g}_\ast^{\pm 1}]$.

Let $\hp_1(S,Q)$ be the profinite completion of $\pi_1(S,Q)$ and $\hp_1(S,Q)/\!\sim$ be the set of conjugacy classes of elements 
of $\hp_1(S,Q)$ and $\cP_2(\hp_1(S,Q)/\!\sim)$ the profinite set of unordered pairs of elements of $\hp_1(S,Q)/\!\sim$. 
Since $\pi_1(S,Q)$ is conjugacy separable, the set $\pi_1(S,Q)/\!\sim$ embeds in the profinite set $\hp_1(S,Q)/\!\sim$. 
We define the set of \emph{profinite simple closed curves} $\hL(S)$ on $S$ to be the closure of the set $\iota(\cL(S))$ inside 
the profinite set $\cP_2(\hp_1(S,Q)/\!\sim)$. An order of the set $\{\alpha,\alpha^{-1}\}$ is preserved by the conjugacy action and
defines an \emph{orientation} for the associated equivalence class $[\alpha^{\pm1}]\in\hL(S)$.

For all $k\geq 0$, there is a natural embedding of the set $C(S)_{k-1}$ of isotopy classes of \emph{$k$-multicurves}, that is to say,
multicurves on $S$ of cardinality $k$, into the profinite set $\cP_k(\hL(S))$
of unordered subsets of $k$ elements of $\hL(S)$. The set of \emph{profinite $k$-multicurves}
on $S$ is then the closure of the set $C(S)_{k-1}$ inside the profinite set $\cP_k(\hL(S))$, for $k>0$.

\begin{definition}\label{geopro}Let $L(\hP(S))$ be the abstract simplicial profinite complex whose $k$-simplices are the profinite $k$-multicurves on
$S$. The abstract simplicial profinite complex $L(\hP(S))$ is called the \emph{complex of profinite curves on $S$}.
\end{definition}

By Theorem~4.2 of \cite{B1}, for a suitable order of the vertex set of $L(\hP(S))$, there is a natural isomorphism:
\begin{equation}\label{realizationprof}
\kC(S)_\bt\cong L(\hP(S))_\bt.
\end{equation}

\subsection{The complex of profinite symmetric curves}
For a hyperbolic marked hyperelliptic surface $(S,\u,\cP)$, by definition, there is a natural $\U(S,\u,\cP)$-equivariant embedding 
of simplicial complexes $C(S,\u,\cP)\subseteq C(S\ssm\cP)$ and hence, for a suitable order of the vertex set, an embedding of 
$\U(S,\cP)^{[4]}$-simplicial sets $C(S,\u,\cP)_\bt\subseteq C(S\ssm\cP)_\bt$. 

In particular, the conditions we need in order to define
the $\hU(S,\cP)$-completion of the simplicial set $C(S,\u,\cP)_\bt$ are satisfied.
The \emph{symmetric profinite curve complex} $\hC(S,\u,\cP)_\bt$ is then defined to be the $\hU(S,\cP)$-completion of the
simplicial set $C(S,\u,\cP)_\bt$. Note that the natural $\U(S,\cP)$-equivariant map $C(S,\u,\cP)_\bt\to\hC(S,\u,\cP)_\bt$ is then also an embedding.

Since, by Theorem~\ref{conghyp}, we have $\hU(S,\cP)=\kU(S,\cP)$ the $\hU(S,\cP)$-completion of $C(S,\u,\cP)_\bt$ coincides with its
$\kU(S,\cP)$-completion. We will show that the natural $\hU(S,\u,\cP)$-equivariant continuous map of simplicial profinite sets
$\hC(S,\u,\cP)_\bt\to \kC(S\ssm\cP)_\bt$ is an embedding.

For $\U^\l$ a finite index subgroup of $\U(S,\cP)$ contained in the abelian level $\U(S,\cP)^{[4]}$, the quotient finite simplicial set 
$C^\l(S,\u,\cP)_\bt:=C(S,\u,\cP)_\bt/\U^\l$ identifies with the nerve of the DM boundary of $\ccM(S,\u,\cP)^\l$.
Therefore, the profinite curve complex $\hC(S,\u,\cP)_\bt$ is naturally isomorphic to the inverse limit of the
nerves of the DM boundaries of all level structures over $\cM(S,\u,\cP)$. This fact implies the following nontrivial description
of stabilizers for the continuous action of $\hU(S,\cP)$ on $\hC(S,\u,\cP)_\bt$ (cf.\ Proposition~6.5 in \cite{PFT}):

\begin{theorem}\label{stabilizerprohyp}Let $\hU^\l$ be an open subgroup of $\hU(S,\cP)$ and let $\U^\l:=\hU^\l\cap\U(S,\cP)$, where we have 
identified $\U(S,\cP)$ with its image in $\hU(S,\cP)$. Then, we have:
\begin{enumerate}
\item For $\s\in C(S,\u,\cP)_\bt\subset\hC(S,\u,\cP)_\bt$, the stabilizer $\hU_\s^\l$ of $\s$ for the action of $\hU^\l$ on $\hC(S,\u,\cP)_\bt$ is 
the closure, inside the profinite group $\hU(S,\cP)$, of the stabilizer $\U_\s^\l$ of $\s$ for the action of $\U^\l$ on $C(S,\u,\cP)_\bt$.
\item The stabilizer $\hU_\s^\l$ is the profinite completion of the stabilizer $\U_\s^\l$.
\end{enumerate}
\end{theorem}

\begin{proof}(i): For $\U^\mu$ a finite index normal subgroup of $\U^\l$ contained in $\U(S,\cP)^{[4]}$, the stabilizer $(\U^\l/\U^\mu)_{\s_\mu}$ of the image 
$\s_\mu$ of $\s$ in the quotient $C^\mu(S,\u,\cP)_\bt$ for the action of the finite group $\U^\l/\U^\mu$ identifies with the Galois group of the normal \'etale 
covering $\wh{\Delta}_\s^\mu\to\wh{\Delta}_\s^\l$, where $\wh{\Delta}_\s^\mu$ and $\wh{\Delta}_\s^\l$ are punctured tubular neighborhoods of 
the open strata parameterized by $\s$ in $\ccM(S,\u,\cP)^\mu$ and $\ccM(S,\u,\cP)^\l$, respectively. Since the fundamental group of 
$\wh{\Delta}_\s^\l$ is isomorphic to the stabilizer $\U^\l_\s$, it follows that $(\U^\l/\U^\mu)_{\s_\mu}=\U^\l_\s/\U^\mu_\s$. 
By taking the inverse limit on all such $\U^\mu$, we get the first item of the theorem. 
\smallskip

\noindent
(ii): It is not restrictive to take $\hU^\l=\hPU(S,\cP)$. For $\cP=\emptyset$, the conclusion then follows from the exact sequences~(\ref{stabPHyp}),
the isomorphisms~(\ref{wreathPMCG}) and~(\ref{PHMCG}), Theorem~4.5 in \cite{B1} and the congruence subgroup property for hyperelliptic  
mapping class groups (cf.\ Theorem~\ref{conghyp}).

For $\cP\neq\emptyset$, given $\s\in C(S,\u,\cP)$, let us denote by $\bar\s$ the, possibly trivial, isotopy class in $S$ of a multicurve on $S\ssm\cP$ 
in the isotopy class $\s$. There is then a natural epimorphism $\PU(S,\cP)_\s\to\PU(S)_{\bar\s}$ whose kernel, by the definition of $\PU(S,\cP)$, coincides
with the kernel of the natural epimorphism $\PG(S,\cP)_\s\to\PG(S)_{\bar\s}$. By the case $\cP=\emptyset$ considered above, 
it is then enough to show that the kernel of the natural epimorphism $\kPG(S,\cP)_\s\to\kPG(S)_{\bar\s}$ is the profinite completion of the kernel of
$\PG(S,\cP)_\s\to\PG(S)_{\bar\s}$. But this immediately follows from Theorem~4.5 in \cite{B1} and the Birman short exact sequence for procongruence
mapping class groups (cf.\ Corollary~4.7 in \cite{congtop}).
\end{proof}

From the description~(\ref{stabHyp}) of the stabilizer $\U(S,\cP)_\s$ of a simplex $\s\in C(S,\u,\cP)$ for the action of the hyperelliptic mapping class group
$\U(S,\cP)$, it then follows:

\begin{corollary}\label{stabilizerprohypdescription}For a simplex $\s\in\hC(S,\u,\cP)_\bt$ in the image of $C(S,\u,\cP)_\bt$, the stabilizer $\hU(S,\cP)_\s$
for the action of the profinite hyperelliptic mapping class group $\hU(S,\cP)$ on $\hC(S,\u,\cP)_\bt$ is described by the short exact sequence:
\begin{equation}\label{stabproHyp}
1\to\prod_{\g\in\s}\tau_\g^{\ZZ}\to\hU(S,\cP)_\s\to\hU(S\ssm\s,\cP),
\end{equation}
where $\tau_\g^{\ZZ}$, $\hU(S,\cP)_\s$ and $\hU(S\ssm\s,\cP)$ are naturally isomorphic to the profinite completions of the groups 
$\tau_\g^{\Z}$, $\U(S,\cP)_\s$ and $\U(S\ssm\s,\cP)$,
respectively.
\end{corollary}

\begin{proof}Right exactness of the sequence~(\ref{stabproHyp}) follows from Theorem~\ref{stabilizerprohyp} and left exactness 
from Theorem~4.5 in \cite{B1}.
\end{proof}

\begin{corollary}\label{stabilizerprohyp2} For a simplex $\s\in\hC(S,\u,\cP)_\bt$, let $\s$ also denote its image in $\kC(S\ssm\cP)$.
Then, for $\hU^\l$ an open subgroup of $\hU(S,\cP)$, we have $\hU_\s^\l=\hU^\l\cap\kG(S,\cP)_\s$, where we have identified $\hU(S,\cP)$
with its image in the procongruence mapping class group $\kG(S,\cP)$.
\end{corollary}

\begin{proof}It is not restrictive to assume that $\hU^\l=\hU(S,\cP)$ and that $\s$ is in the image of $C(S,\u,\cP)_\bt$. 
The conclusion then follows from Corollary~\ref{stabilizerprohypdescription} and 
Theorem~4.9 in \cite{BF}.
\end{proof}

Corollary~\ref{stabilizerprohyp2} implies, in particular, the claim we made at the beginning of this section:

\begin{corollary}\label{injectivity}For a hyperbolic marked hyperelliptic surface $(S,\u,\cP)$, the natural map 
$\hC(S,\u,\cP)_\bt\to \kC(S\ssm\cP)_\bt$ is injective.
\end{corollary}

As done for the procongruence curve complex, we can now give a more intrinsic description of the profinite symmetric curve 
complex $\hC(S,\u,\cP)_\bt$. Let us define the set of \emph{symmetric profinite $k$-multicurves}
on $S$ as the closure of the set $C(S,\u,\cP)_{k-1}$ inside the profinite set $\cP_k(\hL(S\ssm\cP))$, for $k>0$.

\begin{definition}\label{geoprohyp}Let $L(\hP(S\ssm\cP),\u)$ be the abstract simplicial profinite complex whose $k$-simplices are the symmetric 
profinite $k$-multicurves on $S$. The abstract simplicial profinite complex $L(\hP(S\ssm\cP),\u)$ is called the \emph{complex of symmetric 
profinite curves on $S$}.
\end{definition}

By Corollary~\ref{injectivity} and the isomorphism~(\ref{realizationprof}), for a suitable order of the vertex set of $L(\hP(S\ssm\cP),\u)$, 
there is then a natural isomorphism of simplicial profinite sets:
\begin{equation}\label{realizationprofsym}
\hC(S,\u,\cP)_\bt\cong L(\hP(S\ssm\cP),\u)_\bt.
\end{equation}

\subsection{Centralizers of symmetric profinite multitwists and the center of the profinite hyperelliptic mapping class group}
There is a natural $\G(S\ssm\cP)$-equivariant map $\Tau\co C(S\ssm\cP)\to\G(S\ssm\cP)$ which assigns to a multicurve $\s$ 
the multitwist $\Tau(\s):=\prod_{\g\in\s}\tau_\g$.
By the universal property of the $\kG(S\ssm\cP)$-completion, this map then extends to a continuous $\kG(S\ssm\cP)$-equivariant map:
\[\check{\Tau}\co L(\hP(S\ssm\cP))\to\kG(S\ssm\cP),\]
which assigns to a profinite multicurve $\s$ the procongruence multitwist $\check{\Tau}(\s):=\prod_{\g\in\s}\tau_\g$.
It is clear that this map restricts to a continuous $\hU(S,\cP)$-equivariant map:
\[\wh{\Tau}\co L(\hP(S\ssm\cP),\u)\to\hU(S,\cP),\]
which assigns to a symmetric profinite multicurve $\s$ the \emph{symmetric profinite multitwist} $\wh{\Tau}(\s):=\prod_{\g\in\s}\tau_\g$.
Thanks to Corollary~\ref{stabilizerprohyp2}, we can determine the centralizer and the normalizer in $\hU(S,\cP)$ of such an element:

\begin{corollary}\label{centprosymult}For $\s\in L(\hP(S\ssm\cP),\u)$ and all $k\in\N^+$, we have:
\[Z_{\hU(S,\cP)}(\wh{\Tau}(\s)^k)=N_{\hU(S,\cP)}(\langle\wh{\Tau}(\s)^k\rangle)=N_{\hU(S,\cP)}(\langle\tau_\g^k\,|\,\,\g\in\s\rangle)=\hU(S,\cP)_\s,\]
where by $\langle A\rangle$ we denote the closed subgroup generated by a set $A$ of elements in a profinite group and
$\hU(S,\cP)_\s$ is the stabilizer described in Corollary~\ref{stabilizerprohypdescription}. In particular, the map $\wh{\Tau}$ defined above is injective.
\end{corollary}

\begin{proof}Let us identify $\hU(S,\cP)$ with its image in the procongruence mapping class group $\kG(S,\cP)$. It is clear that we have
$Z_{\hU(S,\cP)}(\wh{\Tau}(\s)^k)=Z_{\kG(S,\cP)}(\wh{\Tau}(\s)^k)\cap\hU(S,\cP)$ and the same holds for the normalizers appearing in
the statement of the corollary. But then the conclusion is an immediate consequence of Corollary~\ref{stabilizerprohyp2} and 
Corollary~4.11 in \cite{BF} (see also Corollary~4.3 in \cite{congtop}).
\end{proof}

We can now determine the centralizer of an open subgroup of the profinite hyperelliptic mapping class group:

\begin{theorem}\label{centerhyp}Let $(S,\u,\cP)$ be a marked hyperelliptic surface such that $S$ is hyperbolic and $n(\ring{S}_{/\u}\ssm\cB_\u)>4$. 
Then, for every open subgroup $\hU^\l$ of $\hU(S,\cP)$, the centralizer of $\hU^\l$ in $\hU(S,\cP)$ is generated by the hyperelliptic involution $\u$ if 
$\cP=\emptyset$ and is trivial if $\cP\neq\emptyset$. 
\end{theorem}

\begin{proof}Let us consider first the case $\cP=\emptyset$. Since $S$ is hyperbolic, by the exact sequence~(\ref{HypEx}), there is an exact sequence:
\begin{equation}\label{HypExPro}
1\to\langle\u\rangle\to\hU(S)\to\hG(S_{/\u},\cB_\u)_{p_\u}\to 1,
\end{equation}
where the group $\hG(S_{/\u},\cB_\u)_{p_\u}$ identifies with an open subgroup of $\hG(\ring{S}_{/\u}\ssm\cB_\u)$.

For $n(\ring{S}_{/\u}\ssm\cB_\u)>4$, it is well known that the centralizer in $\hG(\ring{S}_{/\u}\ssm\cB_\u)$ of an open subgroup
is trivial (cf.\ (i) of Theorem~4.14 in \cite{BF}). This then implies the claim of the theorem for $\cP=\emptyset$.

The next case to deal with is when $\cP=\{P\}$ is a single point. Let us consider the profinite Birman short exact sequence:
\[1\to\hp_1(S,P)\to\hU(S,P)\sr{\hat{f}_P}{\to}\hU(S)\to 1.\]
Since, as it is well known, the centralizer in $\hp_1(S,P)$ of the open subgroup $\hU^\l\cap\hp_1(S,P)$ is trivial, the intersection 
$Z_{\hU(S,P)}(\hU^\l)\cap\hp_1(S,P)$ is also trivial. Therefore, the centralizer $Z_{\hU(S,P)}(\hU^\l)$ identifies with a subgroup of 
the centralizer $Z_{\hU(S)}(\hat{f}_P(\hU^\l))$.

By the case $\cP=\emptyset$ of the theorem, $Z_{\hU(S)}(\hat{f}_P(\hU^\l))$ is generated by the hyperelliptic involution $\u$.
Let then $\td\u\in\U(S,P)$ be a hyperelliptic involution which lifts $\u$. By (ii) of Lemma~3.1 in \cite{Magnus}, the only elements of order $2$
in the profinite group $\hp_1(S,P)\cdot\langle\td\u\rangle$ are the conjugates of $\td\u$.
Thus, either $Z_{\hU(S,P)}(\hU^\l)$ is trivial or is generated by a conjugate of $\td\u$.
Now, by (i) of Lemma~3.1 in \cite{Magnus}, the hyperelliptic involution $\td\u$, and so its conjugates, are self-centralizing in the group 
$\hp_1(S,P)\cdot\langle\td\u\rangle$. Therefore, $Z_{\hU(S,P)}(\hU^\l)$ is trivial.

For $\sharp\, \cP\geq 2$, fix a point $Q\in\cP$ and let $\cP':=\cP\ssm Q$. As a first step, let us show that, 
for every open subgroup $\hU^\l$ of $\hU(S,\cP)$, the centralizer of $\hU^\l$ in $\hU(S,\cP)$ is in fact contained in the subgroup 
$\hU(S,\cP',Q)$ of $\hU(S,\cP)$. Let $\g$ be a separating simple closed curve on $S$ bounding on one side a subsurface $S'$ of $S$
containing only the point $Q\in\cP$ and on the other side a subsurface $S''$ containing all the other points $\cP'$. 
Let us also choose $\g$ in a way that the marked hyperelliptic surfaces $(S',\u|_{S'},Q)$ and $(S'',\u|_{S''},\cP')$ are not isomorphic.

The open subgroup $\hU^\l$ contains a power $\tau_\g^k$, for some $k\in\N^+$. By Corollary~\ref{centprosymult}, we have that
$Z_{\hU(S,\cP)}(\tau_\g^k)=\hU(S,\cP)_\g$ and then $Z_{\hU(S,\cP)}(\hU^\l)\subseteq Z_{\hU(S,\cP)}(\tau_\g^k)=\hU(S,\cP)_\g$.
It is easy to check that the image of the subgroup $\hU(S,\cP)_\g$ via the natural epimorphism $\hU(S,\cP)\to\Sigma_{\cP}$
is contained in the stabilizer of the point $Q$ in the symmetric group $\Sigma_{\cP}$ and so the same is true for the image of $Z_{\hU(S,\cP)}(\hU^\l)$.
It follows that $Z_{\hU(S,\cP)}(\hU^\l)$ is contained in $\hU(S,\cP',Q)$, as claimed above.

We now proceed by induction on $\sharp\, \cP\geq 2$, with the base of the induction provided by the case $\sharp\, \cP=1$.
Let us consider the profinite Birman short exact sequence:
\[1\to\hp_1(S,Q)\to\hU(S,\cP',Q)\sr{\hat{f}_Q}{\to}\hU(S,\cP')\to 1.\]
As above, the intersection of $Z_{\hU(S,\cP)}(\hU^\l)$ with the profinite surface group $\hp_1(S,Q)$ is trivial. 
Therefore, the centralizer $Z_{\hU(S,\cP)}(\hU^\l)$ identifies with a subgroup of the centralizer $Z_{\hU(S,\cP')}(\hat{f}_Q(\hU^\l))$. 
Since $\sharp\, \cP'=\sharp\,\cP-1\geq 1$, by the induction hypothesis, the latter group is trivial and so the conclusion of the theorem follows.
\end{proof}

\begin{remark}For $n(\ring{S}_{/\u}\ssm\cB_\u)=3$, the statement of the theorem is still true but mostly trivial. For $n(\ring{S}_{/\u}\ssm\cB_\u)=4$,
the situation is more complicated. By Proposition~2.7 in \cite{FM}, there is an isomorphism:
\[\G(S_{0,4})\cong\PSL_2(\Z)\ltimes(\Z/2\times\Z/2),\]
where $\PSL_2(\Z)$ acts on $\Z/2\times\Z/2$ (the \emph{Klein subgroup}) through its quotient $\PSL_2(\Z/2)$ and the pure mapping class group 
$\PG(S_{0,4})$ identifies with the kernel of this action. We then also have an isomorphism:
\[\hG(S_{0,4})\cong\wh{\PSL_2(\Z)}\ltimes(\Z/2\times\Z/2).\]
In particular, the pure profinite mapping class group $\hPG(S_{0,4})$ is centralized by the Klein subgroup of $\hG(S_{0,4})$ and so
the statement of Theorem~\ref{centerhyp} does not hold in this case.

Another interesting case which is not considered in Theorem~\ref{centerhyp} is when $(S,\u,\cP)$ is a marked hyperelliptic closed surface of genus $1$.  
In this case, $\hU(S,\cP)=\hG(S,\cP)$ so that, by Theorem~4.14 in \cite{BF}, we have that $Z_{\hU(S,\cP)}(\hU^\l)$ is generated by 
the hyperelliptic involution for $\cP\leq 2$ and is trivial for $\cP> 2$.
\end{remark}

\end{document}